\documentclass[12pt, a4paper]{amsart}
\usepackage{amsmath}
\usepackage{amsthm,graphics,tabularx,amssymb,shapepar}
\usepackage{mathtools}
\usepackage{stmaryrd}
\usepackage{fullpage}
\usepackage{enumerate,enumitem}
\usepackage{caption}
\usepackage{tikz-cd}
\usepackage{amscd}
\usepackage[cmyk]{xcolor}
\usepackage{comment}
\excludecomment{review}

\usepackage{hyperref}
\usepackage[
backend=biber,
style=alphabetic,
sorting=nyt,
maxnames=9,
minnames=1,
maxalphanames=99,
]{biblatex}
\makeatletter
\newcommand*{\rom}[1]{\expandafter\@slowromancap\romannumeral #1@}
\makeatother

\newcommand{\BC}{{\mathbb {C}}}

\newcommand{\BH}{{\mathbb {H}}}

\newcommand{\BR}{{\mathbb {R}}}
\newcommand{\BS}{{\mathbb {S}}}

\newcommand{\BW}{{\mathbb {W}}}

\newcommand{\BZ}{{\mathbb {Z}}}

\newcommand{\CE}{{\mathcal {E}}}

\newcommand{\CN}{{\mathcal {N}}}

\newcommand{\CS}{{\mathcal {S}}}
\newcommand{\CT}{{\mathcal {T}}}

\newcommand{\RD}{{\mathrm {D}}}

\newcommand{\RH}{{\mathrm {H}}}

\newcommand{\RU}{{\mathrm {U}}}

\newcommand{\CWR}{{\mathrm{CW}}}

\newcommand{\EP}{{\mathrm{EP}}}

\newcommand{\GL}{{\mathrm{GL}}}

\newcommand{\Hom}{{\mathrm{Hom}}}
\newcommand{\Herm}{{\mathrm{Herm}}}

\renewcommand{\Im}{{\mathrm{Im}}}
\newcommand{\Ind}{{\mathrm{Ind}}}
\newcommand{\ind}{{\mathrm{ind}}}

\newcommand{\Isom}{{\mathrm{Isom}}}

\newcommand{\Ker}{{\mathrm{Ker}}}

\newcommand{\Mp}{{\mathrm{Mp}}}

\renewcommand{\Re}{{\mathrm{Re}}}

\newcommand{\Rep}{{\mathrm{Rep}}}

\newcommand{\SU}{{\mathrm{SU}}}

\newcommand{\Sp}{{\mathrm{Sp}}}
\newcommand{\Stab}{{\mathrm{Stab}}}
\newcommand{\Span}{{\mathrm{Span}}}

\newcommand{\Tr}{\operatorname{Tr}}
\newcommand{\Trd}{\operatorname{Trd}}

\newcommand{\Nrd}{\operatorname{Nrd}}

\newcommand{\oH}{\operatorname{H}}

\newcommand{\g}{\mathfrak g}

\renewcommand{\k}{\mathfrak k}

\renewcommand{\b}{\mathfrak b}

\newcommand{\n}{\mathfrak n}

\renewcommand{\v}{\mathfrak v}

\renewcommand{\t}{\mathfrak t}

\newcommand{\be}{\begin {equation}}
\newcommand{\ee}{\end {equation}}
\newcommand{\bee}{\begin {equation*}}
\newcommand{\eee}{\end {equation*}}

\newcommand{\fre}{{Fr\'{e}chet }}
\newcommand{\kun}{{K\"{u}nneth }}
\newcommand{\poin}{{Poincar\'{e} }}
\newcommand{\smod}{\CS\mathrm{mod}}

\newcommand{\ta}{\prescript{t}{}{a}}

\newcommand{\kX}{\prescript{\kappa}{}{X}}

\DeclareMathOperator{\otimeshat}{\widehat{\otimes}}

\newtheorem{thm}{Theorem}[section]
\newtheorem{cort}[thm]{Corollary}
\newtheorem{lemt}[thm]{Lemma}
\newtheorem{prpt}[thm]{Proposition}

\theoremstyle{remark}

\theoremstyle{remark}

\newcommand{\HX}[1]{{\color{red} {#1}}}

\begin{document}
\title{Casselman--Wallach Property for Homological Theta Lifting}
\author[Z. Geng]{Zhibin Geng}
\address{Department of Mathematics and New Cornerstone Science Laboratory, The University of Hong Kong}
\email{gengzb@hku.hk}

\author{Hang Xue}
\address{Department of Mathematics, The University of Arizona}
\email{xuehang@arizona.edu}

\date{\today}

\keywords{Theta lifting; Corank-one parabolic stable filtration; Casselman--Wallach representations}
\subjclass[2020]{11F27, 22E50}

\begin{abstract}
    In this paper, we establish the Casselman--Wallach property for homological theta lifting over archimedean local fields. As a consequence, the Euler--\poin characteristic is a well-defined element in the Grothendieck group of Casselman--Wallach representations. Our main tool is a corank-one parabolic stable filtration on the Weil representation.
\end{abstract}

\maketitle
\tableofcontents

\section{Introduction}
Let $F$ be an archimedean local field. Let $\BW$ be a finite-dimensional symplectic space over $F$. Fix a non-trivial unitary character $\psi: F \to \BC^{\times}$.
Let $(G', H')$ be a reductive dual pair in $\Sp(\BW)$. 
In the context of theta lifting, one associates to $(G',H')$ a corresponding pair $(G, H)$ 
(roughly speaking, $(G, H)$ differs from $(G', H')$ only in that certain symplectic factors are replaced by their metaplectic double cover; see Subsection \ref{Subsec: ThetaCorrespondence} for the precise statement).
Denote by $\Omega_{G,H,\psi}$ the associated Weil representation of $G \times H$. We may omit the subscript when no confusion arises.

Write $\Rep_G^{\mathrm{CW}}$ (resp.\ $\Rep_H^{\mathrm{CW}}$) for the category of Casselman–Wallach representations of $G$ (resp.\ $H$). For $\pi \in \Rep_G^{\mathrm{CW}}$, let $\pi^{\vee}$ denote its contragredient. Consider the coinvariant space
\[
\Theta_0(\pi) := (\Omega \otimeshat \pi^{\vee})_G.
\]
It should be noted that, a priori, it is not clear whether $\Theta_0(\pi)$ is Hausdorff.
The usual full theta lift $\Theta(\pi)$ of $\pi$ is defined as the maximal Hausdorff quotient of $\Theta_0(\pi)$, namely
\[
\Theta(\pi) := (\Omega \otimeshat \pi^{\vee})_G / \overline{\{0\}}.
\]
For a general introduction to the theory of theta lifting, we refer the reader to \cite{MVW1987Theta} and \cite{GKT2025ThetaBook}. 
A brief overview is also provided in Subsection \ref{Subsec: ThetaCorrespondence}.

A natural question is whether $\Theta_0(\pi)$ is already Hausdorff, so that the quotient by $\overline{\{0\}}$ is unnecessary. An affirmative answer will lead to the homological formulation of the theta lifting (in the sense of Schwartz homology), and naturally raises the further question of whether the higher homology groups are also Casselman--Wallach representations.
See Section \ref{Sec: Schwartz homology} for the basic notions of Schwartz homology.
The main result of this paper provides an affirmative answer to the above question and serves as a preliminary result for the homological theta lifting.
The precise statement is as follows.

\begin{thm}\label{ThmA}
Let $\pi$ be a Casselman--Wallach representation of $G$. For all $i \geq 0$, the Schwartz homology groups
    \[
    \oH_i^{\CS}(G,\;\Omega\otimeshat\pi),
    \]
are Casselman--Wallach representations of $H$, and equal to zero when $i > \dim (G/K)$, where $K$ is a maximal compact subgroup of $G$. In particular, they are Hausdorff.
\end{thm}

By this theorem, the Euler--\poin characteristic
  \[
    \EP_{G}(\Omega \otimeshat \pi):=\sum_{i\ge0}(-1)^i \oH_i^{\CS} (G,\; \Omega\otimeshat\pi)
  \]
is a well-defined element in the Grothendieck group of $\Rep_H^{\mathrm{CW}}$. 
As pointed out in \cite{APS2017EP}, this Euler--\poin characteristic is often more accessible than the full theta lift itself, and is expected to provide insights into the structure of the latter.

The study of homological aspects of theta lifts in the $p$-adic setting was initiated in \cite{APS2017EP}. Our results can be regarded as the corresponding archimedean analogue. In this setting, we work with topological vector spaces in contrast to merely vector spaces in the $p$-adic setting, which gives rise to many additional difficulties and technical complications.

Our main tool is a corank-one parabolic stable filtration on the Weil representation. This is the archimedean counterpart of the filtration first studied by Kudla.
In the archimedean setting, due to the non-exactness of the Jacquet functors and the presence of the transverse derivatives, the general filtration takes a much more complicated shape.
In this paper, we restrict ourselves to the corank-one case, 
which suffices for our purposes and admits a particularly clean form.

We now briefly describe the proof of Theorem~\ref{ThmA}, which contains some of the major innovations of this paper. First, we may assume $(G, H)$ is an irreducible dual pair.
Using Casselman’s subrepresentation theorem together with a lemma from \cite{HT1990AnalyticLocalization} concerning the topology of long exact sequences (see Lemma \ref{Lem: Hausdorff LES}), we reduce Theorem \ref{ThmA} to the following.

\begin{thm}\label{Thm: principal series CW}
    Let $(G, H)$ and $\Omega$ be as before, and let $P_{\min} = L \ltimes N$ be a minimal parabolic subgroup of $G$. Let $V$ be an irreducible representation of $L$. Denote by $\Ind_{P_{\min}}^G (V)$ the corresponding principal series representation. Then, for all $i \geq 0$,
    $$
    \oH^{\CS}_{i}(G, \Omega \otimeshat \Ind_{P_{\min}}^G (V) )
    $$
    is a Casselman--Wallach representation of $H$.
\end{thm}

Frobenius reciprocity quickly reduces Theorem~\ref{Thm: principal series CW} to
    \[
    \oH^{\CS}_{i}(P_{\min}, \Omega \otimeshat V \otimes \delta^{-1}_{P_{\min}}).
    \]
Thus, we need to study the restriction of $\Omega$ to parabolic subgroups. This is where we make use of the corank-one filtration on the Weil representation. 
It allows us to argue by induction on the split rank of $G$.
For Type~I dual pairs, the base case is the compact dual pair, for which Hausdorffness is immediate. 
The finite length property follows from Howe's original work \cite{Howe1989Remarks} and \cite[Section 3]{Howe1989TranscendingClassicalInvTheory}. In fact, in Howe's original proof of Howe duality \cite{Howe1989TranscendingClassicalInvTheory}, compact dual pairs already play a fundamental role in determining the general case.
For Type~II dual pairs, the base case is $(\GL_1,\GL_m)$, which we calculated explicitly in Subsection~\ref{Subsec: (GL1,GLm)}.
For the induction step, we treat Type~I and Type~II dual pairs separately, in Sections \ref{Sec: Type I CW} and \ref{Sec: Type II CW}, respectively. 

\subsection*{Acknowledgments}
ZG thanks Kaidi Wu and Hao Ying for helpful discussions. He also thanks Binyong Sun for his constant support and encouragement.
ZG is partially supported by the New Cornerstone Science Foundation through the New Cornerstone Investigator Program awarded to Professor Xuhua He.
HX is partially supported by the NSF grant DMS \#2154352. 
Some initial discussions related to this work started during the Arizona Winter School 2025.

\section{Preliminaries and Background}\label{Sec: preliminary}
In this section, we recall some preliminary material and fix the notation to be used in the paper.

\subsection{Casselman--Wallach representation}
Let $G$ be a real reductive group with complexified Lie algebra $\g$ and maximal compact subgroup $K$.
A Casselman-Wallach representation of $G$ is a smooth \fre representation of moderate growth and of finite length. The readers may consult \cite[Chapter 11]{Wallach1992Real} or \cite{Bernstein2014Smooth} for details. 
Denote by $\Rep_G^{\mathrm{CW}}$ the category of 
Casselman--Wallach representations of $G$.

By the Casselman--Wallach globalization theorem, the functor of taking $K$-finite vectors induces an equivalence between $\Rep_G^{\mathrm{CW}}$ and the category of Harish--Chandra $(\g,K)$-modules. 
Consequently, morphisms in $\Rep_G^{\mathrm{CW}}$ have well-behaved topological properties, as shown in the following lemma.

\begin{lemt}[{\cite[Corollary 2.2.5]{AGS15Derivative}}]\label{Lem: CW closed image}
    \begin{enumerate}
        \item $\Rep_G^{\mathrm{CW}}$ is an abelian category.
        \item Any morphism in $\Rep_G^{\mathrm{CW}}$ has closed image.
    \end{enumerate}
\end{lemt}

Every Casselman--Wallach representation is in fact a nuclear \fre space, NF-space for short. See \cite[Appendix A]{Casselman2000Bruhat} for more details about NF-spaces.
If $W$ is an NF-space, and $V \subseteq W$ is a closed subspace, then $V$ and $W/V$ are both NF-spaces. Additionally, all surjective morphisms between \fre spaces are open.
\par
For two topological vector spaces $V$ and $W$, 
write $V\otimeshat W$ for their completed projective tensor product.
If $V$ and $W$ are both NF-spaces, so is $V \otimeshat W$.
Furthermore, for a fixed NF-space $V$, the functor $W\mapsto V\otimeshat W$ is exact on the category of NF-spaces.

The following lemma plays an important role in our reduction step.

\begin{lemt}[Casselman's subrepresentation theorem, {\cite[Subsection 3.8]{Wallach1988RRGI}}]
    Let $\pi$ be an irreducible representation of $G$. 
    Write $P_{\min}=L\ltimes N$ for a minimal parabolic subgroup of $G$.
    Then there exists an irreducible finite-dimensional representation $V$ of $L$, such that $\pi$ can be imbedded into $\Ind^G_{P_{\min}}(V)$.
\end{lemt}

\subsection{Theta correspondence}\label{Subsec: ThetaCorrespondence}
In this section, we briefly review the theory of theta lifting and fix notation. For more details, the reader may consult \cite{MVW1987Theta, Adams2007ThetaR, GKT2025ThetaBook}.
Let $F$ be an archimedean local field. Let $\BW$ be a finite-dimensional symplectic space over $F$. Fix a non-trivial unitary character $\psi: F \to \BC^{\times}$.

We first recall basic facts about reductive dual pairs in $\Sp(\BW)$.
Our conventions follow \cite[Chapter~10.5]{GKT2025ThetaBook} and \cite{Adams2007ThetaR}.
Let $(G',H')$ be a reductive dual pair in $\Sp(\mathbb{W})$. Then there is an orthogonal decomposition $\BW = \oplus_{i=1}^k \BW_i$, such that
\[
(G',H')=(G'_{1},H'_{1})\times\cdots\times(G'_{k},H'_{k}),
\]
where each $(G'_{i},H'_{i})$ is an irreducible reductive dual pair in $\Sp(\mathbb{W}_{i})$.
The following tables present a complete list of irreducible reductive dual pairs, together with the associated division algebra and involution $(\RD, \kappa)$ in the Type~I case.
Note that if $(G',H')$ is a reductive dual pair in $\Sp(\BW)$, then so is $(H',G')$.

\begin{table}[ht]
\centering
\label{tab:Rpairs}
\begin{tabular}{c|c c c c}
  & $G'$ & $H'$ & $\mathrm{Sp}(\BW)$ & $(\RD, \kappa)$\\ \hline
Type I & $\mathrm{Sp}(2n,\mathbb{R})$ & $\mathrm{O}(p,q,\mathbb{R})$ & $\mathrm{Sp}(2n(p+q),\mathbb{R})$ & $(\BR, 1)$ \\
  & $\mathrm{U}(r,s)$ & $\mathrm{U}(p,q)$ & $\mathrm{Sp}(2(p+q)(r+s),\mathbb{R})$ & $(\BC, \Bar{\cdot})$ \\
  & $\mathrm{O}^{*}(2n)$ & $\mathrm{Sp}(p,q)$ & $\mathrm{Sp}(4n(p+q),\mathbb{R})$ & $(\BH, \Bar{\cdot})$ \\ \hline
Type II & $\mathrm{GL}(n,\mathbb{R})$ & $\mathrm{GL}(m,\mathbb{R})$ & $\mathrm{Sp}(2nm,\mathbb{R})$ &\\
  & $\mathrm{GL}(n,\mathbb{H})$ & $\mathrm{GL}(m,\mathbb{H})$ & $\mathrm{Sp}(8nm,\mathbb{R})$ &\\
\end{tabular}
\caption*{Irreducible reductive dual pairs over $\mathbb{R}$}
\end{table}

\begin{table}[ht]
\centering
\label{tab:Cpairs}
\begin{tabular}{c|c c c c}
  & $G'$ & $H'$ & $\mathrm{Sp}(W)$ & $(\RD, \kappa)$ \\ \hline
Type I & $\mathrm{Sp}(2n,\mathbb{C})$ & $\mathrm{O}(m,\mathbb{C})$ & $\mathrm{Sp}(2nm,\mathbb{C})$ & $(\BC, 1)$ \\ \hline
Type II & $\mathrm{GL}(n,\mathbb{C})$ & $\mathrm{GL}(m,\mathbb{C})$ & $\mathrm{Sp}(2nm,\mathbb{C})$ & \\
\end{tabular}
\caption*{Irreducible reductive dual pairs over $\mathbb{C}$}
\end{table}

Let
\[
G := G_1 \times \cdots \times G_k, \qquad H := H_1 \times \cdots \times H_k,
\]
where for each $i$ the pair $(G_i,H_i)$ is obtained from $(G'_i,H'_i)$ by the following rule: if neither factor is an odd orthogonal group then $G_i:=G'_i$ and $H_i:=H'_i$, while if one factor is an odd orthogonal group then the other factor is replaced by its metaplectic double cover.
Denote by $\Mp(\BW)^{\BC^1}$ the $\BC^1$-metaplectic cover of $\Sp(\BW)$, where $\BC^1 := \{z \in \BC \mid |z| = 1\}$.
There exists a splitting homomorphism 
\[
s : G \times H \longrightarrow \Mp(\BW)^{\BC^1}.
\]
Let $\Omega_{\psi}$ be the Weil representation of $\Mp(\BW)^{\BC^1}$ with respect to $\psi$. 
Via the above splitting homomorphism $s$, we regard $\Omega_{\psi}$ as a representation of $G \times H$.

Suppose $\pi$ is an irreducible representation of $G$, define the full theta lift of $\pi$ to be
\[
\Theta(\pi):=(\Omega_{\psi} \otimeshat\pi^{\vee})_G/\overline{\{0\}}.
\]
By Howe's pioneering work \cite{Howe1979thetaseries} \cite{Howe1989Remarks} \cite{Howe1989TranscendingClassicalInvTheory}, $\Theta(\pi)$ is of finite length as $H$ representation, and hence belongs to $\Rep_{H}^{\mathrm{CW}}$. 
Moreover, when $\Theta(\pi)$ is non-zero, it admits a unique irreducible quotient.
The analogous statement over $p$-adic fields, formulated by Howe as the Howe duality conjecture, was subsequently established in full generality by Gan--Takeda \cite{GT2016HoweDuality} and Gan--Sun \cite{GS2017HoweDuality}.

\section{Schwartz homology and reduction step}\label{Sec: Schwartz homology}

In this section, we review necessary facts on Schwartz analysis and Schwartz homology.
In Subsection \ref{Subsec: reduction step}, we show that the proof of Theorem~\ref{ThmA} can be reduced to that of Theorem~\ref{Thm: principal series CW}.

\subsection{Schwartz induction}
We work with induced representations in the Schwartz setting.
Let $G$ be an almost linear Nash group, $H \subseteq G$ a Nash subgroup, and $V$ a smooth Fréchet representation of $H$ of moderate growth.
We refer to \cite[Subsection 6.2]{CHEN2021108817} for the definition of the Schwartz induced representation $\ind_H^G V$.
When $H\backslash G$ is compact, it coincides with the usual smooth induced representation.

For a tempered bundle $\CE$ over a Nash manifold $X$, we write $\Gamma^{\CS}(X, \CE)$ for the space of Schwartz sections of $\CE$ over $X$ in the sense of \cite[Definition 6.1]{CHEN2021108817}.
The Schwartz induction $\ind_{H}^{G}V$ admits the following geometric interpretation. 
Let $\mathcal{V}= G\times^{H} V$ be the associated homogeneous bundle. It is a (right) $G$-tempered bundle over $H \backslash G$; see \cite[Subsection 3.3]{CHEN2021108817} and \cite[Subsection 2.1]{Xue20Bessel}. 
By \cite[Proposition 6.7]{CHEN2021108817}, we have
\begin{equation}\label{Eq: geometric of schwartz ind}
    \operatorname{ind}_{H}^{G}V \cong \Gamma^{\mathcal{S}}(H\backslash G,\mathcal{V}).
\end{equation}

To study the filtration for the Type II dual pairs, we need the following generalization of equality~(\ref{Eq: geometric of schwartz ind}).

\begin{lemt}\label{Lem: Schwartz Fibre bundle}
    Let $G$ be an almost linear Nash group and $H \subseteq G$ a closed Nash subgroup. Let $X'$ be a Nash manifold with a right $H$-action. We define an $H$-action on $X' \times G$ by 
    $$
    h \cdot (x', g) =(x'h, h^{-1}g).
    $$
    Let 
    $$X :=  X'\times_H G$$
    be the quotient of $X' \times G$ by this $H$-action,
    which admits a right $G$-action.
    Let $\CE$ be a (right) $G$-tempered bundle over $X$. Denote by $\pi$ the natural action of $G$ on $\Gamma^{\CS}(X, \CE)$, and $\rho$ the action of $H$ on $\Gamma^{\CS}(X', \CE|_{X'})$, then
    $$
    \pi \cong \ind^G_H \rho.
    $$
\end{lemt}

\begin{proof}
    We construct the isomorphism explicitly as follows:
    $$
    \begin{array}{rcl}
    \pi &\longrightarrow& \ind^G_H \rho\\
    \phi &\longmapsto&
    (g \mapsto (\pi(g)\phi)|_{X'})
    \end{array},
    $$

    $$
    \begin{array}{rcl}
    \ind^G_H \rho &\longrightarrow& \pi\\
    f &\longmapsto&
    (x \mapsto r_g(f(g)(x')))
    \end{array},
    $$
    where $x = (x', g)$, and $r_g$ denotes the action of $g \in G$ on $\CE$. 
    One can check that the second map is well-defined and that these two maps are inverses of each other.
\end{proof}

\subsection{Schwartz homology}
The theory of Schwartz homology is developed in \cite{CHEN2021108817}. In this subsection, we introduce the necessary tools that will be used in our subsequent proof. We refer the readers to \cite{CHEN2021108817} for more details.

Let $G$ be an almost linear Nash group. Denote by $\CS\mathrm{mod}_G$ the category of smooth \fre representations of $G$ of moderate growth.
For $V \in \CS\mathrm{mod}_G$, consider the space of $G$-coinvariants 
$$V_G := V/(\sum_{g \in G}(g-1)\cdot V),$$
which is given the quotient topology and becomes a locally convex (not necessarily Hausdorff) topological vector space.
In \cite{CHEN2021108817}, the Schwartz homology groups $\oH^\CS_i(G,-)$ are introduced so that
$$\oH^\CS_0(G,V)\cong V_G,$$
and each $\oH^\CS_i(G,V)$ is naturally a locally convex (not necessarily Hausdorff) topological vector space. 
This homology theory possesses the following favorable properties.

\begin{prpt}[Shapiro's Lemma, {\cite[Theorem 7.5]{CHEN2021108817}}]\label{Prop: Shapiro Lem}
    Let $H$ be a Nash subgroup of $G$ and $V \in \CS\mathrm{mod}_H$.  Then for all $i\in \BZ_{\geq 0}$, there is an identification of topological vector spaces
$$
\oH_{i}^\CS(G,  (\ind_H^G (V \otimes \delta_H))\otimes \delta_G^{-1}) \cong \oH_{i}^\CS(H, V).
$$
Here, $\delta_{G}$ (resp. $\delta_{H}$) denotes the modular character of $G$ (resp. $H$).
\end{prpt}

\begin{prpt}[{\cite[Theorem 7.7]{CHEN2021108817}}]\label{Prop: homology comparison}
For every representation $V$ in $\CS\mathrm{mod}_{G}$ and every $i\in \BZ_{\geq 0}$, there is an identification of topological vector spaces
\[
\oH^{\CS}_{i}(G, V) \cong \oH_{i}(\g, K; V).
\]
Here and henceforth, for $V \in \CS\mathrm{mod}_{G}$, $\oH_{*}(\g, K; V)$ denotes the homology groups of the Koszul complex
\[
\cdots \rightarrow (\wedge^{l+1}(\g/\k) \otimes V)_{K}
\rightarrow (\wedge^{l}(\g/\k) \otimes V)_{K}
\rightarrow (\wedge^{l-1}(\g/\k) \otimes V)_{K}
\rightarrow\cdots
\]
\end{prpt}

\begin{prpt}[{\cite[Corollary 7.8]{CHEN2021108817}}]
Every short exact sequence 
$$0\rightarrow V_{1}\rightarrow V_{2}\rightarrow V_{3}\rightarrow 0$$
in the category 
$\CS\mathrm{mod}_{G}$ yields a long exact sequence
\[
\cdots \rightarrow\oH^{\CS}_{i+1}(G, V_{3})\rightarrow\oH^{\CS}_{i}(G, V_{1})\rightarrow\oH^{\CS}_{i}(G, V_{2})\rightarrow \oH^{\CS}_{i}(G, V_{3})\rightarrow\cdots
\]
of (not necessarily Hausdorff)  locally convex topological vector spaces.
\end{prpt}

It is important in practice to determine the Hausdorffness of the Schwartz homology. We have the following useful propositions.

\begin{prpt}[{\cite[Theorem 5.9, Proposition 5.7]{CHEN2021108817}}]\label{Prop: projective hausdorff}
Let $V$ be a relatively projective representation in $\CS\mathrm{mod}_G$.  Then the coinvariant space $V_G$ is a \fre space.
Moreover, when $G$ is compact, every representation in $\CS\mathrm{mod}_G$ is relatively projective. 
\end{prpt}

\begin{prpt}[{\cite[Lemma 3.4]{borel2000continuous}}, {\cite[Proposition 1.9]{CHEN2021108817}}]\label{Prop: finite hausdorff}
Let $V \in \CS\mathrm{mod}_G$ and $i \in \BZ_{\geq 0}$.  If $\oH_{i}^\CS(G; V)$ is finite-dimensional, then it is Hausdorff.  
\end{prpt}

\begin{prpt}[\kun formula, {\cite[Proposition 3.8]{geng2025Shalika}}]
    Let $G_1$ and $G_2$ be two almost linear Nash groups.
    Assume $V_i \in \CS\mathrm{mod}_{G_i}$, $i = 1,2$, are both NF-spaces. If for all $j \in \BZ_{\geq 0}$, $\oH^{\CS}_{j}(G_i, V_i)$ are NF-spaces, then there is an isomorphism
    $$
    \oH^{\CS}_m(G_1 \times G_2, V_1 \otimeshat V_2) \cong \bigoplus_{p+q=m} \oH^{\CS}_p(G_1, V_1) \otimeshat \oH^{\CS}_q(G_2, V_2), \quad \forall \ m \in \BZ_{\geq 0},
    $$
    as topological vector spaces. In particular, $\oH^{\CS}_m(G_1 \times G_2, V_1 \otimeshat V_2)$ is an NF-space. 
\end{prpt}

\begin{prpt}[Hochschild-Serre spectral sequence for nilpotent normal subgroups]
    Let $H = L \ltimes N$ be an almost linear Nash group with $N$ being a nilpotent normal subgroup. Consider $V \in \CS\mathrm{mod}_H$, if for all $j \in \BZ_{\geq 0}$, $\oH^{\CS}_{j}(N, V) \in \CS\mathrm{mod}_L$, then there exists convergent first quadrant spectral sequences:
    $$
    E^2_{p,q} = \oH^{\CS}_{p}(L, \oH^{\CS}_{q}(N, V)) \implies \oH^{\CS}_{p+q}(H ,V).
    $$
    Furthermore, assume $V \in \smod_{G \times H}$,
    if for all $p,q \in \BZ_{\geq 0}$, $\oH^{\CS}_{p}(L, \oH^{\CS}_{q}(N, V))$ are Casselman--Wallach representations of $G$, then so is $\oH^{\CS}_{i}(H ,V)$ for all $i \in \BZ_{\geq 0}$.
\end{prpt}

\begin{proof}
    The first assertion is exactly \cite[Proposition 3.9]{geng2025Shalika}. The second assertion follows from \cite[Lemma 2.15]{wu2025BZfiltration} together with Lemma \ref{Lem: CW closed image}.
\end{proof}

\begin{lemt}[{\cite[Lemma 6.2.2]{aizenbud2015twisted}}]\label{Lem: Gourevitch Lemma}
    Let X be a Nash manifold and let V be a real vector space. Let $\phi: X\to V^{\ast}$ be a Nash map. Suppose $0 \in V^{\ast}$ is a regular value of $\phi$. Let $\CT(X)$ be the space of tempered functions on $X$. Then $\phi$ induces a map 
        \[
        \chi: V \to \CT(X), \quad \chi(v)(x) = \exp(\pi \sqrt{-1} \Re (\phi(x)(v))),
        \]
    which gives an action of V on $S(X)$ by $\pi(v)(f) := \chi(v) \cdot f$ (the usual multiplication). We have the following assertions.
    \begin{enumerate}[label=(\roman*)]
        \item $\oH_i^{\CS}(\v,S(X)) = 0$ for $i > 0$.
        
        \item Let $X_0 := \phi^{-1}(0)$, which is smooth by the regularity assumption. Denote by $r$ the restriction map $r: S(X) \to S(X_0)$. Then $r$ gives an isomorphism $\oH_0^{\CS}(\v,S(X)) \xrightarrow{\sim} S(X_0)$. 
        In particular, if $X_0$ is empty, then $\oH_0^{\CS}(\v,S(X)) = 0$.
    \end{enumerate}
\end{lemt}

We also need the following two lemmas when studying the filtration on the Weil representation.
\begin{lemt}\label{Lem: homology commute ind}
    Let $G, H$ be two almost linear Nash groups, and $P$ (resp. $Q$) be a Nash subgroup of $G$ (resp. $H$). Take $\pi \in \smod_{P \times Q}$. Assume that $\oH^{\CS}_{i}(P, \pi \otimes \delta_{G/P})$ is Hausdorff for all $i \geq 0$. Then for all $i \geq 0$, there is an identification in $\smod_H$
    $$
    \oH^{\CS}_{i}(G, \ind^{G \times H}_{P \times Q}\pi) \cong \ind^H_Q\oH^{\CS}_{i}(P, \pi \otimes \delta_{G/P}).
    $$
    Here, $\delta_{G/P} := (\delta_{G})|_{P} \cdot \delta_{P}^{-1}$.
\end{lemt}

\begin{proof}
    Let $R_{\bullet} \twoheadrightarrow \pi$ be a $P \times Q$ strong relative projective resolution of $\pi$, with each $R_{i}$ an NF-space. Then $\ind^{G \times H}_{P \times Q} R_{\bullet} \twoheadrightarrow \ind^{G \times H}_{P \times Q} \pi$ is a $G \times H$ strong relative projective resolution of $\ind^{G \times H}_{P \times Q} \pi$. Note that 
    $$
    (\ind^{G \times H}_{P \times Q} R_i)_G = (S(H) \otimeshat R_i \otimes \delta_{G/P} \otimes \delta_{Q}^{-1} )_{P \times Q} = \ind^{H}_{Q} (R_i \otimes \delta_{G/P})_P.
    $$
    Thus $\oH^{\CS}_{i}(G, \ind^{G \times H}_{P \times Q}\pi)$ is the $i$-th homology group of the chain complex
    $$\cdots \xrightarrow{} \ind^{H}_{Q} (R_i \otimes \delta_{G/P} )_P \xrightarrow{} \cdots.$$

    Denote $R_{\bullet}^{'} := R_{\bullet} \otimes \delta_{G/P}$. Consider
    the chain complex
    $$\cdots \xrightarrow{} (R_{i+1}^{'})_P \xrightarrow{d_{i+1}} (R_i^{'})_P \xrightarrow{d_i} (R_{i-1}^{'} )_P \xrightarrow{}\cdots.$$
    Set $K_i := \Ker\, d_i$, $B_i := \Im\, d_{i+1}$, both equipped with subspace topology of $(R_i^{'})_P$.
    Note that $K_i$ is clearly closed in $(R_i^{'})_P$, and $B_i$ is also closed in $(R_i^{'})_P$ since $\oH^{\CS}_{i}(P, \pi \otimes \delta_{G/P})$ is Hausdorff.
    As $\ind^{H}_{Q}$ is an exact functor, we obtain 
    $$
    0 \xrightarrow{} \ind^{H}_{Q} K_i \xrightarrow{} \ind^{H}_{Q} (R_i^{'})_P \xrightarrow{} \ind^{H}_{Q} (R_{i-1}^{'})_P,
    $$
    $$
    0 \xrightarrow{} \ind^{H}_{Q} K_{i+1} \xrightarrow{} \ind^{H}_{Q} (R_{i+1}^{'})_P \xrightarrow{} \ind^{H}_{Q} B_i \xrightarrow{} 0.
    $$
    Hence, 
    $$\oH^{\CS}_{i}(G, \ind^{G \times H}_{P \times Q}\pi) = \frac{\ind^{H}_{Q} K_i}{\ind^{H}_{Q} B_i}.$$
    Moreover, we have
    $$
    0 \xrightarrow{} \ind^{H}_{Q} B_i \xrightarrow{} \ind^{H}_{Q} K_i \xrightarrow{} \ind^{H}_{Q} \oH^{\CS}_{i}(P, \pi \otimes \delta_{G/P}) \xrightarrow{} 0
    $$
    According to the open mapping theorem, it follows that
    $$
    \oH^{\CS}_{i}(G, \ind^{G \times H}_{P \times Q}\pi) \cong \ind^H_Q\oH^{\CS}_{i}(P, \pi \otimes \delta_{G/P}).
    $$
\end{proof}

The following lemma provides us a sufficient condition under which inverse limits and Schwartz homology can be switched.

\begin{lemt}\label{Lem: homological filtration}
    Assume that $V \in \smod_{G \times H}$ admits a $G\times H$ decreasing filtration 
\[
V =: V_0 \supseteq V_1 \supseteq \cdots \supseteq V_k \supseteq \cdots
\]
such that $V \cong \varprojlim_k V / V_k$. Suppose that:
\begin{enumerate}
    \item For all $i$ and $k$, $\oH^{\CS}_{i}(G, V_k / V_{k+1})$ is a Casselman--Wallach representation of $H$;
    \item There exists an integer $K$ such that for all $k > K$ and all $i$, $\oH^{\CS}_{i}(G, V_k / V_{k+1}) = 0$.
\end{enumerate}
Then, for all $i$, we have a natural topological isomorphism
\[
\oH^{\CS}_{i}(G,  \varprojlim_k V / V_{k}) \;\cong\;  \varprojlim_k \, \oH^{\CS}_{i}(G, V / V_{k}),
\]
and the space on either side is a Casselman--Wallach representation of $H$.

\end{lemt}

\begin{proof}
    This follows from \cite[Proposition 13.2.3]{EGA3}, with cochain complexes $K^{*}_{\alpha}$ in \cite[Proposition 13.2.3]{EGA3} replaced by chain complexes $P_{k,\bullet}$ where
\[
P_{k,\bullet}\to V/V_{k}\to 0
\]
is a strong relative projective resolution. Recall that an inverse system indexed by $\BZ_{\geq 0}$
\[
\cdots\to U_{\alpha+1}\to U_{\alpha}\to\cdots
\]
satisfies the Mittag-Leffler condition if for any $\alpha\in \BZ_{\geq 0}$, there is some $\beta\geq\alpha$ so that the image of
$$
U_{\gamma}\to U_{\alpha}
$$
are the same for all $\gamma\geq\beta$. We note that the system
\[
\cdots\to V/V_{k+1}\to V/V_{k}\to\cdots
\]
satisfies the Mittag-Leffler condition since the maps are all surjective. The system
\[
\cdots\to \oH^{\CS}_{i+1}(G,V/V_{k+1})\to \oH^{\CS}_{i+1}(G,V/V_{k})\to\cdots
\]
also satisfies the Mittag-Leffler condition by assumption (2). The proof of \cite[Proposition 13.2.3]{EGA3} can now be copied word by word.
\end{proof}

We conclude this subsection by recalling a lemma concerning the action on coinvariant spaces.

\begin{lemt}[{\cite[Lemma 2.5]{FSX2018godement}}]\label{Lem: GL action on coinv}
    Let $\RD = \BR, \BC$ or $\BH$, and set $G := \GL_n(\RD)$. Then $S(G)$ carries the natural left and right translation actions of $G \times G$. 
    Write $G_l$ for the subgroup $G \times \{1\}$ of $G \times G$.
    For $\pi \in \smod_G$, we have
    $$
    (S(G) \otimeshat \pi)_{G_l} \cong \pi,
    $$
    as representations of $G$.
\end{lemt}

\subsection{Reduction step}\label{Subsec: reduction step}
In this subsection, we reduce the proof of Theorem \ref{ThmA} to the case of the principal series. The following lemma is crucial.

\begin{lemt}[{\cite[Corallory A.11]{HT1990AnalyticLocalization}}]\label{Lem: Hausdorff LES}
    Let $0 \to E \to F \to G \to 0$ be a short exact sequence of complexes of nuclear \fre spaces. Let $A \to B \to C$ be any three consecutive terms in the associated long exact sequence. 
    If C is Hausdorff, then $A \to B$ is a topological homomorphism, i.e., it is continuous and is an open map onto its image. If, in addition, $A \to B$ has closed kernel, then B is Hausdorff. 
\end{lemt}

We are now ready to prove the following proposition, which reduces Theorem~\ref{ThmA} to Theorem~\ref{Thm: principal series CW}.

\begin{prpt}\label{Prop: principal series implies all}
    For $(G,H)$ as in Subsection \ref{Subsec: ThetaCorrespondence}, Theorem~\ref{Thm: principal series CW} implies Theorem~\ref{ThmA}.
\end{prpt}

\begin{proof}
    Let $(G,H)$ be as in Subsection~\ref{Subsec: ThetaCorrespondence}, and denote by $\Omega$ the Weil representation of $G \times H$.

    Let $\pi$ be a Casselman--Wallach representation of $G$. 
    We aim to show that, for all $i \geq 0$, $\oH^{\CS}_{i}(G, \Omega \otimeshat \pi)$ is a Casselman--Wallach representation of $H$.
    We proceed by downward induction on $i$. Let $n = \dim (G/K)$, where $K$ is a maximal compact subgroup of $G$.
    The base case is $i=n+1$, for which all the homology groups vanish.

    Now we do the induction step. Assume that for every $\pi \in \Rep_G^{\CWR}$, 
    $\oH^{\CS}_{k+1}(G, \Omega \otimeshat \pi)$ is a Casselman--Wallach representation of $H$, we claim that, for every $\pi \in \Rep_G^{\CWR}$, $\oH^{\CS}_{k}(G, \Omega \otimeshat \pi)$ is also Casselman--Wallach. 
    We prove this claim by induction on the length of $\pi$. 
    First, assume that $\pi$ is irreducible.
    In this case, by Casselman's subrepresentation theorem, we have a short exact sequence
    \[
    0 \xrightarrow{} \pi  \xrightarrow{} I  \xrightarrow{} \sigma \xrightarrow{} 0,
    \]
    where $I$ is a principal series representation. Then we have
    \begin{equation*}
        \cdots \xrightarrow{} \oH^{\CS}_{k+1}(G, \Omega \otimeshat I)  \xrightarrow{} \oH^{\CS}_{k+1}(G, \Omega \otimeshat \sigma)\xrightarrow{} \oH^{\CS}_{k}(G, \Omega \otimeshat \pi) \xrightarrow{} \oH^{\CS}_{k}(G, \Omega \otimeshat I) \xrightarrow{} \cdots .
    \end{equation*}
    By Theorem~\ref{Thm: principal series CW}, the induction hypothesis, Lemma \ref{Lem: CW closed image} and Lemma~\ref{Lem: Hausdorff LES},
    we deduce that $\oH^{\CS}_{k}(G, \Omega \otimeshat \pi)$ is Casselman--Wallach. 

    Assume the claim is proved for every $\pi$ with length at most $l$. Let $\pi$ be a Casselman--Wallach representation of length $l+1$, and we take a short exact sequence 
    \[
    0 \xrightarrow{} \rho  \xrightarrow{} \pi  \xrightarrow{} \tau \xrightarrow{} 0,
    \]
    such that $\tau$ is irreducible and $\rho$ is of length $l$ (such an exact sequence clearly exists). It gives rise to the long exact sequence
    \[
    \cdots \xrightarrow{} \oH^{\CS}_{k+1}(G, \Omega \otimeshat \tau) \xrightarrow{} \oH^{\CS}_{k}(G, \Omega \otimeshat \rho) \xrightarrow{} \oH^{\CS}_{k}(G, \Omega \otimeshat \pi)  \xrightarrow{} \oH^{\CS}_{k}(G, \Omega \otimeshat \tau)\xrightarrow{} \cdots .
    \]
    Lemma~\ref{Lem: Hausdorff LES} thus implies that $\oH^{\CS}_{k}(G, \Omega \otimeshat \pi)$ is Casselman--Wallach. This concludes the induction step. 
\end{proof}

\section{Type II dual pairs}\label{Sec: Type II CW}
In this section, we prove Theorem \ref{Thm: principal series CW} for the Type~II dual pairs. Our main tool is the corank\mbox{-}one parabolic stable filtration on the Weil representation. 
Using this filtration, we reduce the problem for $(\GL_{n+1},\GL_m)$ to the case $(\GL_n,\GL_m)$, and the basic case $(\GL_1,\GL_m)$ is handled by direct calculation in Subsection \ref{Subsec: (GL1,GLm)}.

We fix the notation for this section. Let $F$ be an archimedean local field and let $\RD$ be a central division algebra over $F$ of dimension $r^2$. Let $\psi:F\to\mathbb{C}^{\times}$ be a non-trivial unitary character of $F$.
For positive integers $n,m$, denote by $M_{n,m}$ the space of $n\times m$ matrices with entries in $\RD$, and set $M_{0,m}=M_{n,0}:=\{0\}$.
Let $\operatorname{Nrd}:M_{n,n}\to F$ be the reduced norm and denote $\nu:=|\operatorname{Nrd}|^r_F$, where $|\cdot|_{F}$ is the normalized absolute value of $F$. 
Let $\operatorname{Trd}: M_{n,n}\to F$ be the reduced trace.
For $x \in M_{n,m}$, we write $\prescript{t}{}{x}$ for the transpose of $x$ when $\RD = \BR$, and conjugate transpose of $x$ when $\RD = \BC$ or $\BH$.

To distinguish the two factors in the Type~II dual pair, we set
$$
G_n:=\GL_n(\RD),\qquad H_m:=\GL_m(\RD),
$$
and also put $G_0 = H_0 :=\{1\}$.
For $k \geq 1$, denote by $S(G_k)$ the space of Schwartz functions on $G_k$ and let $\rho_k$ be the natural action of $G_k\times H_k$ on $S(G_k)$:
\[
(\rho_k(g,h)\phi)(x):=\phi(g^{-1}xh),
\qquad g,x\in G_k,\;h\in H_k,\;\phi \in S(G_n).
\]
Let $S(M_{n,m})$ denote the space of Schwartz functions on $M_{n,m}$.  
Denote by $\sigma_{n,m}$ the natural action of $G_n\times H_m$ on $S(M_{n,m})$, which is given by
\[
(\sigma_{n,m}(g,h)\phi)(x):=\phi(g^{-1}xh),
\qquad g\in G_n,\;h\in H_m,\;x\in M_{n,m},\;\phi \in S(M_{n,m}).
\]
The Weil representation $\Omega_{n,m}$ of the dual pair $(G_n, H_m)$ is given by
\[
\Omega_{n,m}(g,h):=\nu(g)^{-\frac{m}{2}}\,\nu(h)^{\frac{n}{2}}\,\sigma_{n,m}(g,h).
\]

For $0\leq t\leq n$, let $P_{t,n-t}$ (or $P_t$ when no confusion arises) denote the standard parabolic subgroup of $G_n$ attached to the partition $(t,n-t)$.  
We write the Levi decomposition as $P_t=L_t\ltimes N_t$ with $L_t = G_t \times G_{n-t}$.  
Similarly, we define the standard parabolic subgroup $Q_k$ of $H_m$ for $0\leq k\leq m$.

\subsection{The case of \texorpdfstring{$(\GL_1, \GL_m)$}{(GL1, GLm)}}\label{Subsec: (GL1,GLm)}
The case of $(\GL_1, \GL_m)$ is our basic case in the induction step.
In this subsection, we explicitly treat the full theta lift for this case.
To simplify notation, we use the action $\sigma_{1,m}$. 
Since $\Omega_{1,m}$ differs from $\sigma_{1,m}$ only by a character, this causes no loss of generality.

Consider $\RD = \BR,\ \BC,\ \BH.$
Note that 
\begin{equation}\label{Eq: D decomposition}
    \BR^{\times} = \{\pm 1\} \times \BR^{\times}_{>0}, \ \BC^{\times} = \RU(1) \times \BR^{\times}_{>0}, \ \BH^{\times} = \SU(2) \times \BR^{\times}_{>0},
\end{equation}
where the first factor is always compact. 

Let $V := M_{1,m}$, $U := M_{1,m} \setminus \{0\}$. 
Set $r = \dim_{\BR}(V)$.
Then $U$ is a principal $\BR^{\times}_{>0}$-bundle over the sphere $\BS^{r-1}$.
Hence, by \cite[Proposition~5.2]{CHEN2021108817}, $S(U)$ is a relative projective representation of $\BR^{\times}_{>0}$.

Let $\chi$ be a character of $\BR^{\times}_{>0}$.
Consider the short exact sequence
    \[
    0 \xrightarrow{} S(U)  \xrightarrow{} S(V)  \xrightarrow{} \BC \llbracket V \rrbracket \xrightarrow{} 0,
    \]
then we have the associated long exact sequence
    \begin{multline}\label{Eq: GL1 GLm RLES}
    0 \xrightarrow{} \oH^{\CS}_{1}(\BR^{\times}_{>0}, S(V) \otimes \chi) \xrightarrow{} \oH^{\CS}_{1}(\BR^{\times}_{>0}, \BC \llbracket V \rrbracket \otimes \chi)  \xrightarrow{\delta} \oH^{\CS}_{0}(\BR^{\times}_{>0}, S(U) \otimes \chi) \\
    \xrightarrow{} \oH^{\CS}_{0}(\BR^{\times}_{>0}, S(V) \otimes \chi) \xrightarrow{} \oH^{\CS}_{0}(\BR^{\times}_{>0}, \BC \llbracket V \rrbracket \otimes \chi) \xrightarrow{} 0.
    \end{multline}

For $k \in \BZ_{\geq 0}$, define a character $\chi_k$ of $\BR^{\times}_{>0}$ by $\chi_k(x)=x^{k}$.
When $\chi \neq \chi_k$ for any $k \in \BZ_{\geq 0}$, using Borel's Lemma \cite[Proposition 8.2, Proposition 8.3, and Theorem 8.5]{CHEN2021108817}, one can easily show that 
$$\oH^{\CS}_{i}(\BR^{\times}_{>0}, \BC \llbracket V \rrbracket \otimes \chi) = 0, \quad \forall i \in \BZ_{\geq 0}.
$$
For the character $\chi_k$, we have the following lemma.
\begin{lemt}\label{Lem: GL1 GLm closed orbit finite dim}
    Both spaces $\oH^{\CS}_{1}(\BR^{\times}_{>0}, \BC \llbracket V \rrbracket \otimes \chi_k)$ and $\oH^{\CS}_{0}(\BR^{\times}_{>0}, \BC \llbracket V \rrbracket \otimes \chi_k)$ are finite dimensional. 
    Indeed, they coincide and are equal to the space of homogeneous polynomials of total degree $k$:
    $$
    \Span_{\BC} \{ \prod_{i=1}^r x_i^{a_i} | \sum_{i=1}^r a_i = k, a_i \in \BZ_{\geq 0}\},
    $$
    where $x_1,\dots, x_r$ denote the $\BR$-coordinates on $V$.
\end{lemt}

\begin{proof}
    The lemma follows by direct calculation. 
\end{proof}

Using Lemma~\ref{Lem: Hausdorff LES} together with the long exact sequence (\ref{Eq: GL1 GLm RLES}), we obtain the following corollary.

\begin{cort}\label{Cor: R H0 Hausdorff}
    $\oH^{\CS}_{0}(\BR^{\times}_{>0}, S(V) \otimes \chi)$ is Hausdorff. 
\end{cort}

\begin{proof}
    According to \cite[Proposition 5.5]{CHEN2021108817}, $S(U) \otimes \chi$ is a relative projective representation of $\BR^{\times}_{>0}$. Hence, by Proposition \ref{Prop: projective hausdorff}, $\oH^{\CS}_{0}(\BR^{\times}_{>0}, S(U) \otimes \chi)$ is Hausdorff.

    Consider the exact sequence (\ref{Eq: GL1 GLm RLES})
    $$
    \oH^{\CS}_{1}(\BR^{\times}_{>0}, \BC \llbracket V \rrbracket \otimes \chi)  \xrightarrow{\delta} \oH^{\CS}_{0}(\BR^{\times}_{>0}, S(U) \otimes \chi) 
    \xrightarrow{} \oH^{\CS}_{0}(\BR^{\times}_{>0}, S(V) \otimes \chi) \xrightarrow{} \oH^{\CS}_{0}(\BR^{\times}_{>0}, \BC \llbracket V \rrbracket \otimes \chi).
    $$
    According to Lemma \ref{Lem: GL1 GLm closed orbit finite dim} and the discussion preceding it, the first and the last terms are finite-dimensional, and hence Hausdorff by Proposition \ref{Prop: finite hausdorff}. Since $\Im \, \delta$ is finite-dimensional, it is a closed subspace of $\oH^{\CS}_{0}(\BR^{\times}_{>0}, S(U) \otimes \chi)$. 
    Therefore, the corollary follows from Lemma \ref{Lem: Hausdorff LES}.
\end{proof}

We now turn to the $\oH^{\CS}_{1}$ term.
\begin{lemt}\label{Lem: GL1GLm R higher vanish}
    $$
    \oH^{\CS}_{1}(\BR^{\times}_{>0}, S(V) \otimes \chi) = 0
    $$
\end{lemt}

\begin{proof}
    Consider the exact sequence (\ref{Eq: GL1 GLm RLES}).
    When $\chi \neq \chi_k$, the lemma follows immediately from the discussion preceding Lemma \ref{Lem: GL1 GLm closed orbit finite dim}. 
    Hence, it remains to treat the case $\chi = \chi_k$. 
    It suffices to verify that the connecting map $\delta$ in (\ref{Eq: GL1 GLm RLES}) is injective, which will be done by a diagram chase:
    \[
    \begin{tikzcd}
        & \t \otimes S(V) \otimes \chi_k \arrow[r] \arrow[d] & \t \otimes \BC \llbracket V \rrbracket \otimes \chi_k  \\
        S(U) \otimes \chi_k \arrow[r] & S(V) \otimes \chi_k &      
    \end{tikzcd}.
    \]
    The map $\oH^{\CS}_{1}(\BR^{\times}_{>0}, \BC \llbracket V \rrbracket \otimes \chi_k) \xrightarrow{\delta} \oH^{\CS}_{0}(\BR^{\times}_{>0}, S(U) \otimes \chi) = S(\BS^{r-1})$ is then given by
    $$
    \prod_{i=1}^r x_i^{d_i} \mapsto F(a_1,\cdots, a_r) = -\frac{1}{k!} \prod_{i=i}^r a_i^{d_i}.
    $$
    Since the image of every element is a homogeneous polynomial, it cannot vanish identically. Thus, the connecting map $\delta$ in (\ref{Eq: GL1 GLm RLES}) is injective.
\end{proof}

\begin{cort}\label{Cor: GL1GLm higher vanish}
    Let $\pi \in\Rep_{G_1}^{\mathrm{CW}}$. Then for every $i\geq 1$, one has
    $$
    \oH^{\CS}_{i}(\RD^{\times}, S(V) \otimes \pi) = 0.
    $$
\end{cort}

\begin{proof}
    Using the standard spectral sequence argument with respect to the group decomposition (\ref{Eq: D decomposition}), this follows from Lemma \ref{Lem: GL1GLm R higher vanish} and Proposition \ref{Prop: projective hausdorff}.
\end{proof}

Hence, for $\pi \in\Rep_{G_1}^{\mathrm{CW}}$, there is a long exact sequence
    \begin{multline}\label{Eq: GL1 GLm DLES}
        0 \xrightarrow{} \oH^{\CS}_{1}(\RD^{\times}, \BC \llbracket V \rrbracket \otimes \pi)  \xrightarrow{} \oH^{\CS}_{0}(\RD^{\times}, S(U) \otimes \pi)\xrightarrow{} \oH^{\CS}_{0}(\RD^{\times}, S(V) \otimes \pi) \\
        \xrightarrow{} \oH^{\CS}_{0}(\RD^{\times}, \BC \llbracket V \rrbracket \otimes \pi) \xrightarrow{} 0.
    \end{multline}
Based on Lemma~\ref{Lem: GL1 GLm closed orbit finite dim} and the preceding discussion, the spectral sequence arising from the decomposition~\eqref{Eq: D decomposition} shows that the first and last terms are finite-dimensional.

\begin{prpt}\label{Prop: GL1GLm H0}
    Let $\pi\in\Rep_{G_1}^{\mathrm{CW}}$. Then for every $m\geq 1$, one has
    \[
    \oH_{0}^{\CS}(\RD^{\times},S(V)\otimeshat\pi) \in \Rep_{H_m}^{\mathrm{CW}}.
    \]
\end{prpt}

\begin{proof}
    Denote by $\overline{Q_1}$ the opposite parabolic subgroup of $Q_1$ in $H_m$.
    Note that
    $$
    \oH^{\CS}_{0}(\RD^{\times}, S(U) \otimes \pi) = \oH^{\CS}_{0}(\RD^{\times}, \ind_{G_1 \times \overline{Q_1}}^{G_1 \times H_m}( \rho_1 \boxtimes \BC )\otimes \pi) = \ind_{\overline{Q_1}}^{H_m}( \pi \boxtimes \BC),
    $$
    where the first equality is obtained by direct calculation, and the second is a consequence of Lemmas \ref{Lem: homology commute ind} and \ref{Lem: GL action on coinv}.
    The proposition now follows from Lemma \ref{Lem: Hausdorff LES}, the long exact sequence (\ref{Eq: GL1 GLm DLES}), and the discussion following (\ref{Eq: GL1 GLm DLES}).
\end{proof}

Combining Corollary \ref{Cor: GL1GLm higher vanish} with Proposition \ref{Prop: GL1GLm H0}, we complete the proof of Theorem \ref{ThmA} for dual pairs $(\GL_1, \GL_m)$.

\subsection{Corank-one filtration}\label{Subsec: corank1 Filtration Type II}
In this subsection, we introduce the corank\mbox{-}one filtration for Type~II dual pairs, namely a $P_1 \times H_m$-stable filtration on the Weil representation.

We write elements of $M_{n,m}$ in column-block form as $\begin{bmatrix} a\\ b \end{bmatrix}$, where $a \in M_{1,m}, b \in M_{n-1,m}$.
For $f\in S(M_{n,m})$, we define the partial Fourier transform in the $M_{1,m}$-variable by
\[
\widehat f\left(\begin{bmatrix} a\\ b\end{bmatrix}\right)
:=\int_{M_{1,m}} f\left(\begin{bmatrix} x\\ b\end{bmatrix}\right)
\psi(\Trd(\ta x))\,dx.
\]
Conjugating the natural action \(\sigma_{n,m}\) by this partial Fourier transform yields an equivalent model $\sigma^*_{n,m}$ of the \(G_n\times H_m\)-action on \(S(M_{n,m})\).
This realization is more convenient for our purposes.  Below, we record several explicit formulas for \(\sigma^*_{n,m}\).  
\begin{equation}\label{Eq: Type II reductive action}
    \sigma^{*}_{n,m}\left(\begin{bmatrix}
    g_1 & \\ & g_2
    \end{bmatrix},h\right) f\left(\begin{bmatrix}
    a\\b
    \end{bmatrix}\right) = \nu(g_1)^{m} \nu(h)^{-1} f\left(\begin{bmatrix}
    \prescript{t}{}{g}_1 \cdot a \cdot \prescript{t}{}{h}^{-1}\\g_2^{-1} \cdot b \cdot h
    \end{bmatrix}\right),
\end{equation}
where $g_1 \in G_1$, $g_2 \in G_{n-1}$, and $h \in H_m$.

\begin{equation}\label{Eq: Type II unipotent action}
    \sigma^{*}_{n,m}\left(\begin{bmatrix}
    1&u\\ &1
    \end{bmatrix}, 1\right) f\left(\begin{bmatrix}
    a\\b
    \end{bmatrix}\right) = \psi(\Trd(b \cdot \ta \cdot u)) f\left(\begin{bmatrix}
    a\\b
    \end{bmatrix}\right),
\end{equation}
where $u \in M_{1, n-1}$.

Let $M := M_{n,m}$, 
$Z := \left\{
\begin{bmatrix}
0\\b
\end{bmatrix}
\in M_{n,m} | b \in M_{n-1,m}\right\}$, $U := M\setminus Z$.
The space $S(M)$ contains a $P_1 \times H_m$ stable subspace $S(U)$.
Set $S_{Z}(M) := S(M) / S(U)$. Then we have 
    \begin{equation}\label{Eq: Corank1 filtration II}\tag{F-II}
        0 \xrightarrow{} S(U)  \xrightarrow{} S(M)  \xrightarrow{} S_{Z}(M) \xrightarrow{} 0.
    \end{equation}
Denote $V := M_{1,m}$ and view it as a real manifold. Note that $S_{Z}(M) \cong \BC \llbracket V \rrbracket \otimeshat S(M_{n-1,m})$.

Denote by $\n_1$ the complexified Lie algebra of $N_1$.
We first study the $\n_1$-homology of the open piece.
The following lemma holds.

\begin{lemt}\label{Lem: Type II filtration open}
    \[
    \oH_{i}(\n_1, S(U))
    =
    \left\{
    \begin{array}{ll} 
    \ind^{L_1 \times H_m}_{L_1 \times Q_1}(\lambda_{1,1} \otimes \rho_1 \otimeshat \sigma_{n-1, m-1}) & \text{if } i=0, \\
    0 & \text{otherwise},
    \end{array}
    \right.
    \]
    Here:
    \begin{itemize}
        \item $G_1 \times H_1$ acts on $\rho_1$;
        \item $G_{n-1} \times H_{m-1}$ acts on $\sigma_{n-1,m-1}$;
        \item $\lambda_{1,1}$ is the character defined by
                \[\lambda_{1,1} =
                    \begin{cases}
                    \nu^{m} & \text{on } G_{1} \\
                    \nu^{-1} & \text{on } Q_{1} 
                    \end{cases}.\]
    \end{itemize}
\end{lemt}

\begin{proof}
    Define
    $$
    \CN := \left\{
        \begin{bmatrix}
        a\\b
        \end{bmatrix}
        \in M_{n,m} | a \in M_{1,m}, b \in M_{n-1,m}, b \cdot \ta = 0\right\}.
    $$
    Consider the map 
    $$
    \begin{array}{rrcl}
    \phi: & U &\to& N_1^*\\
    &\begin{bmatrix}
    a\\b
    \end{bmatrix} &\mapsto&
    (u \mapsto \Trd(b \cdot \ta \cdot u))
    \end{array}.
    $$
    One can verify directly that $0$ is a regular value of $\phi$, and $\phi^{-1}(0) = \CN \setminus Z$.
    Following Lemma \ref{Lem: Gourevitch Lemma}, we have
    \[
    \oH_{i}(\n_1, S(U))
    =
    \left\{
    \begin{array}{ll} 
    S(\CN \setminus Z)& \text{if } i=0, \\
    0 & \text{otherwise}.
    \end{array}
    \right.
    \]
    Then, together with Lemma \ref{Lem: Schwartz Fibre bundle}, the proposition follows from the same formal calculation as in \cite[Section 3]{Minguez2008typeII}.
\end{proof}

Now we turn to the closed piece.
Recall that $S_{Z}(M) \cong \BC \llbracket V \rrbracket \otimeshat S(M_{n-1,m})$. It admits a decreasing filtration by degree $S_{Z}(M)_k := \BC \llbracket V \rrbracket_{\geq k} \otimeshat S(M_{n-1,m})$. The associated graded pieces are given by 
\begin{equation*}\label{Eq: closed filtration II}
    S_{Z}(M)_k / S_{Z}(M)_{k+1} = F_k \otimes S(M_{n-1,m}),
\end{equation*}
where $F_k$ is the space of homogeneous polynomials on $V$ of total degree $k$.

\begin{lemt}\label{Lem: Type II filtration closed}
    As a representation of $P_1 \times H_m$,
    \begin{equation*}
    S_{Z}(M)_k / S_{Z}(M)_{k+1} \cong \lambda_{1,0} \otimes F_k \otimes \sigma_{n-1,m}.
    \end{equation*}
     Here:
    \begin{itemize}
        \item $N_1$ acts trivially;
        \item $G_{n-1} \times H_{m}$ acts on $\sigma_{n-1,m}$;
        \item $F_k$ carries a $G_{1} \times H_{m}$-action induced from the following action on $V$:
        \[
        (g_1,h)\cdot v = \prescript{t}{}{g}_1^{-1} \cdot v \cdot \prescript{t}{}{h}, \quad v \in V;
        \]
        \item $\lambda_{1,0}$ is the character defined by
                \[\lambda_{1,0} =
                    \begin{cases}
                    \nu^{m} & \text{on } G_{1} \\
                    \nu^{-1} & \text{on } H_{m} 
                    \end{cases}.\]
    \end{itemize}
\end{lemt}

\begin{proof}
    Using the formulas (\ref{Eq: Type II reductive action}) and (\ref{Eq: Type II unipotent action}), the lemma follows from a direct calculation.
\end{proof}

\subsection{Proof of Theorem \ref{Thm: principal series CW}: Type II case}
In this subsection, we prove the following proposition, and use it to prove Theorem \ref{Thm: principal series CW} at the end.

\begin{prpt}\label{Prop: Type II principal induction}
    Let $n \geq 2$, and let $\chi$ be a Casselman--Wallach representation of $G_1$.
    Assume that for all $i$, $m$, and every $\tau \in \Rep_{G_{n-1}}^{\CWR}$,
    \begin{equation}\label{Eq: prop condition type II}\tag{C-II}
        \oH^{\CS}_{i}(G_{n-1}, \Omega_{n-1,m} \otimeshat \tau) \text{ is a Casselman--Wallach representation of } H_m.
    \end{equation}
    Then for all $i$, $m$, and every $\pi \in \Rep_{G_{n-1}}^{\CWR}$,
    $$
    \oH^{\CS}_{i}(G_n, \Omega_{n,m} \otimeshat \Ind_{P_1}^{G_n}(\chi \boxtimes \pi))
    $$
    is also a Casselman--Wallach representation of $H_m$.
\end{prpt}

Following Shapiro's lemma, we know that
$$
\oH^{\CS}_{i}(G_{n}, \Omega_{n,m} \otimeshat \Ind^{G_{n+1}}_{P_1}(\chi \boxtimes \pi)) \cong \oH^{\CS}_{i}(P_1, \Omega_{n,m} \otimeshat (\chi \boxtimes \pi) \otimes \delta_{P_1}^{-1}).
$$
Consider the short exact sequence (\ref{Eq: Corank1 filtration II}) in Subsection \ref{Subsec: corank1 Filtration Type II},
    \begin{equation*}
        0 \xrightarrow{} S(U)  \xrightarrow{} S(M)  \xrightarrow{} S_{Z}(M) \xrightarrow{} 0.
    \end{equation*}
According to Lemma \ref{Lem: Hausdorff LES}, Proposition \ref{Prop: Type II principal induction} follows once we show that, under the condition (\ref{Eq: prop condition type II}),
$$\oH^{\CS}_{i}(P_1, S(U) \otimeshat (\chi \boxtimes \pi) \otimes \delta_{P_1}^{-1}) \quad \text{and} \quad \oH^{\CS}_{i}(P_1, S_{Z}(M) \otimeshat (\chi \boxtimes \pi) \otimes \delta_{P_1}^{-1})
$$ 
are Casselman--Wallach representations for all $i \geq 0$.
These will be proved in Lemmas~\ref{Lem: Type II open orbit}–\ref{Lem: Type II closed orbit}.
Throughout these lemmas, we assume the condition~(\ref{Eq: prop condition type II}).

We begin with the open piece.

\begin{lemt}\label{Lem: Type II open orbit}
    For all $i \geq 0$,
    $\oH^{\CS}_{i}(P_1, S(U) \otimeshat (\chi \boxtimes \pi)\otimes \delta_{P_1}^{-1})$ is a Casselman--Wallach representation of $H_m$.
\end{lemt}

\begin{proof}
    Using Lemma \ref{Lem: Type II filtration open} and the spectral sequence argument, we obtain
    $$
    \oH^{\CS}_{i}(P_1, S(U) \otimeshat (\chi \boxtimes \pi)\otimes \delta_{P_1}^{-1}) \cong \oH^{\CS}_{i}(L_1, \ind^{L_1 \times H_m}_{L_1 \times Q_1}(\lambda_{1,1} \otimes \rho_1 \otimeshat \sigma_{n-1, m-1}) \otimeshat (\chi \boxtimes \pi)\otimes \delta_{P_1}^{-1}),
    $$
    provided that the right-hand side is a Casselman--Wallach representation of $H_m$. According to Lemma \ref{Lem: homology commute ind}, we have
    $$
    \oH^{\CS}_{i}(L_1, \ind^{L_1 \times H_m}_{L_1 \times Q_1}(\lambda_{1,1} \otimes \rho_1 \otimeshat \sigma_{n-1, m-1}) \otimeshat (\chi \boxtimes \pi)\otimes \delta_{P_1}^{-1}) \cong 
    \ind^{H_m}_{Q_1} \oH^{\CS}_{i}(L_1, \rho_1 \otimeshat \Omega_{n-1, m-1} \otimeshat (\chi \boxtimes \pi)\otimes \eta),
    $$
    where $\eta$ is a certain product of modular character and normalization character. Since it has no influence on our result, we just denote it by $\eta$ to simplify the notation. We shall denote its restriction to $G_{1}$ (resp. $G_{n-1}$) by $\eta_{1}$ (resp. $\eta_{n-1}$).
    According to the \kun formula, we have
    $$
    \oH^{\CS}_{i}(L_1, \rho_1 \otimeshat \Omega_{n-1, m-1} \otimeshat (\chi \boxtimes \pi)\otimes \eta)
    \cong \bigoplus_{p+q=i} 
    \oH^{\CS}_{p}(G_1, \rho_1 \otimes \chi \otimes \eta_1) 
    \otimeshat 
    \oH^{\CS}_{q}(G_{n-1},  \Omega_{n-1,m-1}\otimeshat  \pi  \otimes \eta_{n-1} ).
    $$
    By the condition (\ref{Eq: prop condition type II}), the right-hand side of this equation is a Casselman--Wallach representation of $H_1 \times H_{m-1}$. 
    Therefore, all the representations appearing above are Casselman--Wallach representations of $H_m$.
\end{proof}

Now we turn to the closed piece.
As in Lemma \ref{Lem: Type II filtration closed}, $S_{Z}(M)$ admits a decreasing filtration $S_{Z}(M)_k$ with graded pieces 
$$
S_{Z}(M)_k / S_{Z}(M)_{k+1} \cong \lambda_{1,0} \otimes F_k \otimes \sigma_{n-1,m},
$$
where $F_k$ is the space of homogeneous polynomials on $V$ of total degree $k$.

\begin{lemt}\label{Lem: GL closed grading vanish}
    For all $i$ and $k$, 
    $$\oH^{\CS}_{i}(P_1, (\chi \boxtimes \pi) \otimeshat (S_{Z}(M)_k / S_{Z}(M)_{k+1}) \otimes \delta_{P_1}^{-1} )$$
    is a Casselman--Wallach representation of $H_m$. 
    Moreover, there exists an integer $K$ such that for all $k > K$ and all $i$,  
    $$\oH^{\CS}_{i}(P_1, (\chi \boxtimes \pi) \otimeshat (S_{Z}(M)_k / S_{Z}(M)_{k+1}) \otimes \delta_{P_1}^{-1} ) = 0.$$
\end{lemt}

\begin{proof}    
    By the spectral sequence, it suffices to prove the statement for $$\oH^{\CS}_{i}(L_1,  \oH_{j}(\n_1, \lambda_{1,0} \otimes F_k \otimes \sigma_{n-1,m}) \otimeshat (\chi \boxtimes \pi)\otimes \delta_{P_1}^{-1}).$$
    Let $\xi_{j} \boxtimes \beta_{j}$ be an irreducible $G_1 \times G_{n-1}$ constituent of $\oH_{j}(\n_1, \BC)$. By Lemma \ref{Lem: Type II filtration closed} and the long exact sequence argument, we further reduce to
    \begin{multline*}
        \oH^{\CS}_{i}(L_1,  (\chi \boxtimes \pi) \otimeshat (\xi_{j} \boxtimes \beta_{j}) \otimes \lambda_{1,0} \otimes F_k \otimes \sigma_{n-1,m} \otimes \delta_{P_1}^{-1}) \cong \\
    \bigoplus_{p+q=i}
    \oH^{\CS}_{p}(G_1,  \xi'_{j} \otimes \chi  \otimes F_k ) 
    \otimes
    \oH^{\CS}_{q}(G_{n-1},  \beta'_{j} \otimes \pi \otimeshat \Omega_{n-1,m}),
    \end{multline*}
    where $\beta'_{j}$ (resp. $\xi'_{j}$) denotes $\beta_{j}$ (resp. $\xi_{j}$) twisted by certain inessential character.
    By the condition (\ref{Eq: prop condition type II}), $\oH^{\CS}_{q}(G_{n-1},  \beta'_{j} \otimes \pi \otimeshat \Omega_{n-1,m})$ is a Casselman--Wallach representation of $H_m$.
    Meanwhile, $\oH^{\CS}_{p}(G_1,  \xi'_{j} \otimes \chi  \otimes F_k )$ is finite-dimensional.
    Consequently, 
    $$\oH^{\CS}_{i}(L_1,  (\chi \boxtimes \pi) \otimeshat (\xi_{j} \boxtimes \beta_{j}) \otimes \lambda_{1,0} \otimes F_k \otimes \sigma_{n-1,m} \otimes \delta_{P_1}^{-1})$$
    is a Casselman--Wallach representation of $H_m$.

    Moreover, since $\xi'_{j} \otimes \chi$ is fixed, there must exist a constant $K_{i,j}$, such that for all $k > K_{i,j}$, 
    the $\BR^{\times}_{>0}$-component of $G_1$ acts on $\xi'_{j} \otimes \chi  \otimes F_k$ via a nontrivial character.
    Hence, for all $k > K_{i,j}$, 
    $$\oH^{\CS}_{i}(L_1,  \oH_{j}(\n_1, \lambda_{1,0} \otimes F_k \otimes \sigma_{n-1,m}) \otimeshat (\chi \boxtimes \pi)\otimes \delta_{P_1}^{-1}) = 0.$$
    Let $K:= \max\{K_{i,j}\}$. The lemma follows.
\end{proof}

\begin{lemt}\label{Lem: Type II closed orbit}
    For all $i \geq 0$,
    $\oH^{\CS}_{i}(P_1, S_Z(M) \otimeshat (\chi \boxtimes \pi) \otimes \delta_{P_1}^{-1})$ is a Casselman--Wallach representation of $H_m$.
\end{lemt}

\begin{proof}
    It follows directly from Lemmas \ref{Lem: homological filtration} and \ref{Lem: GL closed grading vanish}.
\end{proof}

\begin{proof}[Proof of Proposition \ref{Prop: Type II principal induction}]
    Using Lemma \ref{Lem: Hausdorff LES}, the proposition follows from Lemmas \ref{Lem: Type II open orbit} and \ref{Lem: Type II closed orbit}.
\end{proof}

We are now ready to prove Theorem \ref{Thm: principal series CW} for the Type II case.

\begin{proof}[Proof of Theorem \ref{Thm: principal series CW}: Type II case]
We prove that Theorem~\ref{Thm: principal series CW} holds for all dual pairs \((G_n,H_m)\) by induction on \(n\).  The base case \((\GL_1,\GL_m)\) has been treated in Subsection~\ref{Subsec: (GL1,GLm)}.

For the induction step, fix $n \geq 2$ and assume that for all $m$, Theorem \ref{Thm: principal series CW} holds for dual pairs $(G_{n-1}, H_m)$. 
By Proposition~\ref{Prop: principal series implies all}, it follows that Theorem~\ref{ThmA} holds for all pairs $(G_{n-1}, H_m)$, so the hypothesis of Proposition~\ref{Prop: Type II principal induction} is satisfied. 
Therefore, via induction by stages, the induction step follows from Proposition~\ref{Prop: Type II principal induction}. The theorem follows.
\end{proof}

\section{Type I dual pairs}\label{Sec: Type I CW}
In this section, we prove Theorem \ref{Thm: principal series CW} for the Type~I dual pairs. Our main tool remains the corank\mbox{-}one parabolic stable filtration on the Weil representation.

Although the same idea also works, we do not treat irreducible dual pairs containing an odd orthogonal factor here. The issue of the associated double cover does not affect our final result but complicates the notation; we leave the details of this case to the interested reader.

We fix the notation for this section. Our conventions are mainly taken from \cite[Section~16.2]{GKT2025ThetaBook}. Although \cite[Section~16.2]{GKT2025ThetaBook} treats only the non-archimedean case, the formula for the mixed model remains valid in the archimedean setting.

Let $F$ be an archimedean local field and let $\RD$ be a division algebra over $F$ with center $E$, together with an involution $\kappa$ such that $E^{\kappa} = F$.
For $d \in \RD$, write $\Bar{d}:=\kappa(d)$.
Let $\psi:F\to\mathbb{C}^{\times}$ be a non-trivial unitary character of $F$. 
Let $\Tr_{\RD/F}: \RD \to F$ be the trace map, defined by $\Tr_{\RD/F}(x)=x$ if $\RD =F$, and $\Tr_{\RD/F}(x)=x + \Bar{x}$ if $\RD \neq F$.
For a $\RD$\mbox{-}module $X$, let $\operatorname{Nrd}: \GL(X) \to E^{\times}$ be the reduced norm and denote $\nu:=|\operatorname{Nrd}|^r_E$, where $|\cdot|_{E}$ is the normalized absolute value of $E$, and $r = \sqrt{\dim_{E} \RD}$.

For a right $\RD$-module $X$, there is a canonical way to define a left $\RD$-module $\kX$. Namely, we set $\kX := X$ as an abelian group, and let $\RD$ act on it from the left as
$$
d \cdot x := x\Bar{d},\quad d \in \RD, \quad x \in X.
$$
Note that we have $\GL(\kX) = \GL(X)$ even as sets of maps on $X$.

Fix $\epsilon \in \{\pm 1\}$.
Let
\[
(G_{n},H_{m})=(G(W_{n}),H(V_{m}))
\]
be an irreducible Type I reductive dual pair, where $(W_{n}, q_W)$ is a $-\epsilon$-Hermitian right $\RD$\mbox{-}module of $\dim_{\RD}W_{n}=n$ and $(V_{m}, q_V)$ an $\epsilon$-Hermitian left $\RD$\mbox{-}module of $\dim_{\RD}V_{m}=m$. 
We assume that neither is an odd orthogonal group.

We fix a splitting data $(\chi,\psi)$ for $G_{n}$ (relative to $V_{m}$) and one $(\mu,\psi)$ for $H_{m}$ (relative to $W_{n}$), where $\chi$, $\mu$ are unitary characters of $E^{\times}$. 
Using this splitting data, we associate the Weil representation 
$\Omega_{n,m}$ to $G_{n}\times H_{m}$.

\subsection{Corank-one filtration}\label{Subsec: corank1 Filtration Type I}

In this subsection, we introduce the corank\mbox{-}one filtration for the Type~I dual pairs, namely a $P_1 \times H_m$-stable filtration on the Weil representation. Here, $P_1$ is the parabolic subgroup of $G_n$ defined in (\ref{Eq: P1 parabolic}) below.

We introduce notation for the case where both $W_n$ and $V_m$ are isotropic.
Let $x_1, x_1^{*}$ be isotropic vectors in $W_n$ such that $q_W(x_1, x_1^{*}) = 1$, and let $y_1, y_1^{*}$ be isotropic vectors in $V_m$ such that $q_V(y_1, y_1^{*}) = 1$.
Set
\[
X_{1}:=\Span_{\RD}\{x_{1}\}, \quad X^{*}_{1}:=\Span_{\RD}\{x^{*}_{1}\}, \quad
Y_{1}:=\Span_{\RD}\{y_{1}\}, \quad Y^{*}_{1}:=\Span_{\RD}\{y^{*}_{1}\}.
\]
Then we have the standard decompositions
\begin{equation}\label{Eq: Witt decomposition}
    W_{n}=X_{1} \oplus W_{n-2} \oplus X^{*}_{1}\quad\text{and}\quad V_{m}=Y_{1} \oplus V_{m-2} \oplus Y^{*}_{1},
\end{equation}
where $W_{n-2}$ (resp. $V_{m-2}$) is the orthogonal complement of $X_{1} \oplus X^{*}_{1}$ (resp. $Y_{1} \oplus Y^{*}_{1}$).

Let
\begin{equation}\label{Eq: P1 parabolic}
    P_{1}=L_{1}N_{1}\subseteq G_{n}
\end{equation}
be the maximal parabolic subgroup of $G_{n}$ stabilizing $X_{1}$.
The elements in $L_{1}$ and $N_{1}$ are, respectively, of the form
\[
\ell(a,g)=\begin{pmatrix}a\\ &g\\ &&(a^{*})^{-1}\end{pmatrix}\quad\text{and}\quad n(b,d)=\begin{pmatrix}1&b&d-\frac{1}{2}bb^*\\ &1&-b^*\\ &&1\end{pmatrix},
\]
where $g \in G(W_{n-2})$, $a \in \GL(X_1)$, $b \in \Hom(W_{n-2}, X_1)$, and $d \in \Herm(X_1^*, X_1)$, with 
$$\Herm(X_1^*, X_1):= \{d \in \Hom(X_1^*, X_1)| d^* = -d\}.$$
Here, the elements $a^* \in \GL(X_1^*)$, $b^* \in \Hom(X_1^*, W_{n-2})$, and $d^* \in \Hom(X_1^*, X_1)$ are defined by requiring that, for $w \in W_{n-2}$,
$$
q_W(ax_1, x_1^*) = q_W(x_1, a^* x_1^*), \ \ 
q_W(bw, x_1^*) = q_W(w, b^* x_1^*), \ \ 
q_W(dx_1^*, x_1^*) = q_W(x_1^*, d^* x_1^*).
$$
Denote
\[
N_{1,0}:=\left\{n(0,d)=\begin{pmatrix}1&&d\\ &1\\ &&1\end{pmatrix} : d\in \Herm(X_1^*, X_1) \right\}.
\]
Thus we have
\[
L_{1}=\GL(X_{1})\times G(W_{n-2}),\quad N_1 = \Hom(W_{n-2}, X_1) \times N_{1,0}.
\]
Similarly, we let
\[
Q_{1}=M_{1}U_{1}\subseteq H_{m}
\]
be the maximal parabolic subgroup of $H_{m}$ stabilizing $Y_{1}$, where $M_{1}=\GL(Y_{1})\times H(V_{m-2})$.

We work in the mixed model. The Weil representation admits the realization 
$$\Omega_{n,m} = S(X^*_1 \otimes V_m) \otimeshat \Omega_{n-2, m} = S(X^*_1 \otimes V_m,\ \Omega_{n-2, m}),$$
where $\Omega_{n-2, m}$ is the Weil representation of $G(W_{n-2}) \times H(V_m)$ with respect to the splitting data $(\chi, \mu, \psi)$.
We now describe the action explicitly. 
Let $\varphi \in S(X^*_1 \otimes V_m,\ \Omega_{n-2, m})$, and $x = x_1^* \otimes v \in X^*_1 \otimes V_m$.
\begin{itemize}
    \item $(\Omega_{n,m}(h)\varphi)(x) = \Omega_{n-2,m}(h)\varphi(h^{-1}x), \quad h \in H(V_{m}), $
    \item $(\Omega_{n,m}(l(1,g_0))\varphi)(x) = \Omega_{n-2,m}(g_0)\varphi(x), \quad g_0 \in G(W_{n-2}), $
    \item $(\Omega_{n,m}(l(a,1))\varphi)(x) = \chi_V(\Nrd a)\nu(a)^{\frac{m}{2}}\varphi(a^*x), \quad a \in \GL(X_1), $
    \item $(\Omega_{n,m}(n(b,0))\varphi)(x) = \rho_0((b^*x, 0))\varphi(x), \quad b \in \Hom(W_{n-2}, X_1), $
    \item $(\Omega_{n,m}(n(0,d))\varphi)(x) = \psi\left(\frac{1}{2}\Tr_{\RD/F}(q_W(x_1^*, -dx_1^*)\overline{q_V(v,v)}) \right)\varphi(x), \   d \in \Herm(X_1^*, X_1). $
\end{itemize}
Here, $\rho_0$ denotes the Heisenberg representation of the Heisenberg group $\RH(W_{n-2}\otimes V_{m})$ realized on the space of $\Omega_{n-2,m}$, and $\chi_V$ is the character associated with the splitting data, as defined in \cite[Formula (12.5)]{GKT2025ThetaBook}.

When $V_m$ is isotropic, we further realize $\Omega_{n-2, m}$ using the mixed model,
$$\Omega_{n-2,m} = S(W_{n-2} \otimes Y^*_1) \otimeshat \Omega_{n-2, m-2} = S(W_{n-2} \otimes Y^*_1,\ \Omega_{n-2, m-2}),$$
where $\Omega_{n-2, m-2}$ is the Weil representation of $G(W_{n-2}) \times H(V_{m-2})$ with respect to the splitting data $(\chi, \mu, \psi)$.
We describe the action explicitly. 
Let $\varphi_0 \in S(W_{n-2} \otimes Y^*_1,\ \Omega_{n-2, m-2})$, and $x = w \otimes y_1^* \in W_{n-2} \otimes Y^*_1$.
\begin{itemize}
    \item $(\Omega_{n-2,m}(g_0)\varphi_0)(x) = \Omega_{n-2, m-2}(g_0)\varphi_0(g_0^{-1}x), \quad g_0 \in G(W_{n-2})$, 
    \item $(\Omega_{n-2,m}(m(1,h_0))\varphi_0)(x) = \Omega_{n-2, m-2}(h_0)\varphi_0(x), \quad h_0 \in H(V_{m-2})$, 
    \item $(\Omega_{n-2,m}(m(a,1))\varphi_0)(x) = \chi_W(\Nrd a)\nu(a)^{\frac{n-2}{2}}\varphi_0(a^*x), \quad a \in \GL(Y_1)$, 
    \item $(\Omega_{n-2,m}(u(b,0))\varphi_0)(x) = \rho_{00}((b^*x, 0))\varphi_0(x), \quad b \in \Hom(V_{m-2}, Y_1)$, 
    \item $(\Omega_{n-2,m}(u(0,d))\varphi_0)(x) = \psi\left(\frac{1}{2}\Tr_{\RD/F}(q_W(w, w)\overline{q_V(y_1^*, -dy_1^*)}) \right)\varphi_0(x), \ d \in \Herm(Y_1^*, Y_1)$,
\end{itemize}
where $\rho_{00}$ denotes the Heisenberg representation of the Heisenberg group $\RH(W_{n-2}\otimes V_{m-2})$ realized on the space of $\Omega_{n-2,m-2}$, $\chi_W$ is the character associated with the splitting data, as defined in \cite[Formula (12.5)]{GKT2025ThetaBook}.

Let $M := X_1^* \otimes V_m$, $Z := \{0\} \subseteq M$, $U := M \setminus Z$. 
The space $S(M)\otimeshat \Omega_{n-2, m}$ contains a $P_1 \times H_m$ stable subspace $S(U)\otimeshat \Omega_{n-2, m}$.
Set $S_{Z}(M) := S(M) / S(U)$. Then we have 
    \begin{equation}\label{Eq: Corank1 filtration I}\tag{F-I}
        0 \xrightarrow{} S(U)\otimeshat \Omega_{n-2, m}  \xrightarrow{} S(M)\otimeshat \Omega_{n-2, m}  \xrightarrow{} S_{Z}(M)\otimeshat \Omega_{n-2, m} \xrightarrow{} 0.
    \end{equation}
Note that $S_Z(M) \cong \BC\llbracket M\rrbracket$, where we view $M$ as a real manifold.

We first study the $\n_1$-homology of the open orbit piece.
Let $\Isom(\kX_1, Y_1)$ be the space of invertible $\RD$-linear map from $\kX_1$ to $Y_1$.
The following lemma holds.
\begin{lemt}\label{Lem: Type I corank1 filtration open orbit}
    If $V_m$ is anisotropic, then 
    $$
    \oH_{i}(\n_1, S(U)\otimeshat \Omega_{n-2, m})
    =0\quad \forall\ i \in \BZ_{\geq 0}.
    $$

    If $V_m$ is isotropic, we have
    \[
    \oH_{i}(\n_1, S(U)\otimeshat \Omega_{n-2, m})
    =
    \left\{
    \begin{array}{ll} 
    \ind^{L_1 \times H_m}_{L_1 \times Q_1} (\lambda_{1,1} \otimes \rho_1 \otimeshat \Omega_{n-2, m-2})& \text{if } i=0, \\
    0 & \text{otherwise}.
    \end{array}
    \right.
    \]
    Here:
    \begin{itemize}
        \item $\rho_1$ is the natural action of $\GL(X_1) \times \GL(Y_1)$ on $S(\Isom(\kX_1, Y_1))$ as
        $$
        ((g,h)\cdot f)(x) := f(g^{-1} x h),
        $$
        for $(g,h) \in \GL(X_1) \times \GL(Y_1)$, $f \in S(\Isom(\kX_1, Y_1))$ and $x \in \Isom(\kX_1, Y_1)$;
        \item $\Omega_{n-2,m-2}$ is the Weil representation of $G(W_{n-2}) \times H(V_{m-2})$ with respect to the splitting data $(\chi, \mu, \psi)$;
        \item $\lambda_{1,1}$ is the character defined by
            \[\lambda_{1,1} =
                \begin{cases}
                \chi_V  \cdot \nu^{\frac{m}{2}} & \text{on } \GL(X_1) \\
                \chi_W \cdot \nu^{\frac{n-2}{2}} & \text{on } \GL(Y_1) 
                \end{cases},\]
                where $\chi_V$ (resp. $\chi_W$) is regarded as a character of $\GL(X_1)$ (resp. $\GL(Y_1)$) via the reduced norm $\Nrd$.
    \end{itemize} 
\end{lemt}

\begin{proof}
    Consider the map
    $$
    \begin{array}{rrcl}
     \phi: &U &\to& N_{1,0}^*\\
    &x_1^* \otimes v &\mapsto&
    \left( d \mapsto \Tr_{\RD/F}(q_W(x_1^*, -dx_1^*)\overline{q_V(v,v)}) \right).
    \end{array}
    $$
    When $V_m$ is anisotropic, $\phi$ has no zeros. 
    By Lemma \ref{Lem: Gourevitch Lemma} and the $N_{1,0}-$action on the mixed model, we obtain
    \[
    \oH_{i}(\n_{1,0}, S(U)\otimeshat \Omega_{n-2, m}) = 0, \quad \forall\ i \in \BZ_{\geq 0}.
    \]
    Now assume that $V_m$ is isotropic, and define 
    \[
    \CN :=\left\{ x_1^* \otimes v\in X_1^* \otimes V_{m} : q_V(v,v)=0 \right\}.
    \]
    One can verify directly that $0$ is a regular value of $\phi$, and $\phi^{-1}(0) = \CN \setminus \{0\}$.
    Applying Lemma \ref{Lem: Gourevitch Lemma} again, we obtain
    \[
    \oH_{i}(\n_{1,0}, S(U)\otimeshat \Omega_{n-2, m})
    =
    \left\{
    \begin{array}{ll} 
    S(\CN \setminus \{0\}) \otimeshat \Omega_{n-2, m}& \text{if } i=0, \\
    0 & \text{otherwise}.
    \end{array}
    \right.
    \]
    
    Note that $x_1^* \otimes y_1 \in \CN \setminus \{0\}$.
    Consider the action of $L_1 \times H_m$ on $\CN \setminus \{0\}$, we have an isomorphism
    $$
    S(\CN \setminus \{0\}) \otimeshat \Omega_{n-2, m} \cong \ind_{(L_1 \times Q_1) \ltimes (N_1/N_{1,0})}^{(L_1 \times H_m) \ltimes (N_1/N_{1,0})} (\lambda' \otimes \rho_1 \otimeshat S(W_{n-2} \otimes Y_1^*) \otimeshat \Omega_{n-2,m-2}).
    $$
    The actions on the right hand side are described as follows.
    \begin{itemize}
        \item $N_1/N_{1,0}$ acts only on $S(W_{n-2} \otimes Y_1^*)$ by
        $$
        n(b,0)\phi(w\otimes y_1^*) = \psi(\Tr_{\RD/F}(q_W(w, w_b)))\phi(w\otimes y_1^*), \quad \text{where } w_b:= b^*x_1^*;
        $$
        \item $\rho_1$ is the natural action of $\GL(X_1) \times \GL(Y_1)$ on $S(\Isom(\kX_1, Y_1))$ as
        $$
        ((g,h)\cdot f)(x) := f(g^{-1} x h),
        $$
        for $(g,h) \in \GL(X_1) \times \GL(Y_1)$, $f \in S(\Isom(\kX_1, Y_1))$ and $x \in \Isom(\kX_1, Y_1)$;
        \item $G(W_{n-2}) \times Q_1$ acts on $\Omega_{n-2,m} = S(W_{n-2} \otimes Y^*_1) \otimeshat \Omega_{n-2, m-2}$ via the restriction of the Weil representation;
        \item $\lambda'$ is a character on $\GL(X_1)$ defined by $\lambda' = \chi_V  \cdot \nu^{\frac{m}{2}}$.
    \end{itemize}

    Consider the map
    $$
    \begin{array}{rrcl}
    \psi: & W_{n-2t} \otimes Y^*_1 &\to& \Hom(W_{n-2}, X_1)^*\\
    & w \otimes y^*_1 &\mapsto&
    (b \mapsto \Tr_{\RD/F}(q_W(w, w_b)))
    \end{array}, \quad \text{where } w_b:= b^*x_1^*,
    $$
    which is an isomorphism.
    Using Lemma \ref{Lem: Gourevitch Lemma} once more, together with Lemma \ref{Lem: homology commute ind}, we obtain 
    \[
    \oH^{\CS}_{i}(N_1/N_{1,0}, S(\CN \setminus \{0\}) \otimeshat \Omega_{n-2, m})
    =
    \left\{
    \begin{array}{ll} 
    \ind_{L_1 \times Q_1}^{L_1 \times H_m} (\lambda_{1,1} \otimes \rho_1 \otimes  \Omega_{n-2,m-2})& \text{if } i=0, \\
    0 & \text{otherwise}.
    \end{array},
    \right.
    \]
    Here the actions are as described in the lemma.
\end{proof}

Now we turn to the closed orbit. Recall that $S_Z(M) \cong \BC\llbracket M\rrbracket$. It admits a decreasing filtration by degree $S_Z(M)_k = \BC\llbracket V_m\rrbracket_{\geq k}$. The associated graded piece 
\begin{equation*}
    F_k := S_Z(M)_k /S_Z(M)_{k+1}
\end{equation*}
is the space of homogeneous polynomials on $M$ of total degree $k$.

\begin{lemt}\label{Lem: Type I corank1 filtration closed orbit}
    As a representation of $P_1 \times H_m$,
    \begin{equation*}
        (S_Z(M)_k  /S_Z(M)_{k+1}) \otimes \Omega_{n-2, m} \cong \lambda_{1,0} \otimes F_k \otimes \Omega_{n-2, m}.
    \end{equation*}
    Here:
    \begin{itemize}
        \item $N_1$ acts trivially;
        \item $\Omega_{n-2,m}$ is the Weil representation of $G(W_{n-2}) \times H_m$ with respect to the splitting data $(\chi, \mu, \psi)$;
        \item $F_k$ carries a $\GL(X_1) \times H_{m}$-action induced from the natural $\GL(X_1) \times H_{m}$-action on $M$;
        \item $\lambda_{1,0} := \chi_V \cdot \nu^{\frac{m}{2}}$ is a character of $\GL(X_1)$.
    \end{itemize}
\end{lemt}

\begin{proof}
    Using the explicit formulas in the mixed model, we observe that $N_{1}$ acts trivially on $M$, hence it acts trivially on $F_k$.
    The same formulas also show that, when restricted to the point
    $0 \in M$, $N_{1}$ acts trivially on $\Omega_{n-2, m}$.
    The remaining assertions follow directly from the mixed model formulas.
\end{proof}

\subsection{Proof of Theorem \ref{Thm: principal series CW}: Type I case}
In this subsection, we prove the following proposition and use it to prove Theorem \ref{Thm: principal series CW} for the Type I case at the end.

Let $(W_n, V_m)$ be as before, and denote by $\t_{W}$ (resp. $\t_{V}$) the Witt tower containing $W_n$ (resp. $V_m$).

\begin{prpt}\label{Prop: Type I principal induction}
    Assume that $W_n$ is isotropic. Let $W_n = X_1 \oplus W_{n-2} \oplus X^{*}_1$ be the decomposition as in (\ref{Eq: Witt decomposition}). 
    Let $\chi$ be a Casselman--Wallach representation of $\GL(X_1)$.
    Assume that for all $i$, for all $V_m \in \t_{V}$, and every $\tau \in \Rep_{G(W_{n-2})}^{\CWR}$,
    \begin{equation}\label{Eq: prop condition type I}\tag{C-I}
        \oH^{\CS}_{i}(G(W_{n-2}), \Omega_{n-2,m} \otimeshat \tau) \text{ is a Casselman--Wallach representation of } H(V_m).
    \end{equation}
    Then for all $i$, for all $V_m \in \t_{V}$, and every $\pi \in \Rep_{G(W_{n-2})}^{\CWR}$,
    $$
    \oH^{\CS}_{i}(G(W_n), \Omega_{n,m} \otimeshat \Ind_{P_1}^{G(W_n)}(\chi \boxtimes \pi)).
    $$
    is also a Casselman--Wallach representation of $H(V_m)$.    
\end{prpt}

Following Shapiro's lemma, we know that
$$
\oH^{\CS}_{i}(G_{n}, \Omega_{n,m} \otimeshat \Ind^{G_{n+1}}_{P_1}(\chi \boxtimes \pi)) \cong \oH^{\CS}_{i}(P_1, \Omega_{n,m} \otimeshat (\chi \boxtimes \pi) \otimes \delta_{P_1}^{-1})
$$
Consider the short exact sequence (\ref{Eq: Corank1 filtration I}) in Subsection \ref{Subsec: corank1 Filtration Type I},
    \begin{equation*}
        0 \xrightarrow{} S(U)\otimeshat \Omega_{n-2, m}  \xrightarrow{} S(M)\otimeshat \Omega_{n-2, m}  \xrightarrow{} S_{Z}(M)\otimeshat \Omega_{n-2, m} \xrightarrow{} 0.
    \end{equation*}
According to Lemma \ref{Lem: Hausdorff LES}, Propsition \ref{Prop: Type I principal induction} follows once we show that, under the condition (\ref{Eq: prop condition type I}),  
$$\oH^{\CS}_{i}(P_1, S(U)\otimeshat \Omega_{n-2, m}   \otimeshat (\chi \boxtimes \pi) \otimes \delta_{P_1}^{-1}) \ \  \text{and} \ \  \oH^{\CS}_{i}(P_1, S_{Z}(M)\otimeshat \Omega_{n-2, m}   \otimeshat (\chi \boxtimes \pi) \otimes \delta_{P_1}^{-1})
$$ 
are Casselman--Wallach representations for all $i \geq 0$.
These will be proved in Lemmas~\ref{Lem: Type I open orbit}–\ref{Lem: Type I closed orbit}.
Throughout these lemmas, we assume the condition~(\ref{Eq: prop condition type I}).

We begin with the open orbit.

\begin{lemt}\label{Lem: Type I open orbit}
    For all $i \geq 0$,
    $\oH^{\CS}_{i}(P_1, S(U) \otimeshat \Omega_{n-2, m}  \otimeshat (\chi \boxtimes \pi)\otimes \delta_{P_1}^{-1})$ is a Casselman--Wallach representation of $H_m$.
\end{lemt}

\begin{proof}
    When $V_m$ is anisotropic, this lemma follows trivially from Lemma~\ref{Lem: Type I corank1 filtration open orbit}.

    Now we assume $V_m$ is isotropic.
    Using Lemma~\ref{Lem: Type I corank1 filtration open orbit} again and the spectral sequence argument, we obtain
    $$
    \oH^{\CS}_{i}(P_1, S(U) \otimeshat (\chi \boxtimes \pi)\otimes\delta_{P_1}^{-1}) \cong \oH^{\CS}_{i}(L_1,\ind^{L_1 \times H_m}_{L_1 \times Q_1} (\lambda_{1,1} \otimes \rho_1 \otimeshat \Omega_{n-2, m-2}) \otimeshat (\chi \boxtimes \pi)\otimes\delta_{P_1}^{-1}),
    $$
    provided that the right-hand side is a Casselman--Wallach representation of $H_m$.
    According to Lemma \ref{Lem: homology commute ind}, we have
    \begin{multline*}
    \oH^{\CS}_{i}(L_1, \ind^{L_1 \times H_m}_{L_1 \times Q_1} (\lambda_{1,1} \otimes \rho_1 \otimeshat \Omega_{n-2, m-2}) \otimeshat (\chi \boxtimes \pi)\otimes\delta_{P_1}^{-1}) \cong \\
    \ind^{H_m}_{Q_1} \oH^{\CS}_{i}(L_1,  (\rho_1 \otimeshat \Omega_{n-2, m-2}) \otimeshat (\chi \boxtimes \pi) \otimes \eta)
    \end{multline*}
    where $\eta$ is a certain product of modular character and normalization character. Since it has no influence on our result, we just denote it by $\eta$ to simplify the notation. We denote its restriction to $\GL(X_{1})$ (resp. $G(W_{n-2})$) by $\eta_{1}$ (resp. $\eta_{n-2}$).
    According to the \kun formula, we have
    \begin{multline*}
    \oH^{\CS}_{i}(L_1,  (\rho_1 \otimeshat \Omega_{n-2, m-2}) \otimeshat (\chi \boxtimes \pi)\otimes \eta) \cong \\
    \bigoplus_{p+q=i} \oH^{\CS}_{p}(\GL(X_1),  \rho_1 \otimes \chi \otimes \eta_1) 
    \otimeshat 
    \oH^{\CS}_{q}(G(W_{n-2}),  \Omega_{n-2, m-2}\otimeshat  \pi  \otimes \eta_{n-2} )
    \end{multline*}
    By the condition~(\ref{Eq: prop condition type I}), the right-hand side of this equation is a Casselman--Wallach representation of $\GL(Y_1) \times H(V_{m-2})$. 
    Therefore, all the representations appearing above are Casselman--Wallach representations of $H_m$.
\end{proof}

Now, we turn to the closed orbit.
As in Lemma \ref{Lem: Type I corank1 filtration closed orbit}, $S_{Z}(M)\otimeshat \Omega_{n-2, m}$ admits a decreasing filtration $S_{Z}(M)_k\otimeshat \Omega_{n-2, m}$ with graded pieces 
$$
(S_Z(M)_k  /S_Z(M)_{k+1}) \otimes \Omega_{n-2, m} \cong \lambda_{1,0} \otimes F_k \otimes \Omega_{n-2, m}.,
$$
where $F_k$ is the space of homogeneous polynomials on $V$ of total degree $k$.
\begin{lemt}\label{Lem: Type I grading closed}
    For all $i$ and $k$, 
    $$\oH^{\CS}_{i}(P_1, (S_Z(M)_k  /S_Z(M)_{k+1}) \otimes \Omega_{n-2, m} \otimeshat (\chi \boxtimes \pi) \otimes\delta_{P_1}^{-1} )$$
    is a Casselman--Wallach representation of $H_m$. 
    Moreover, there exists an integer $K$ such that for all $k > K$ and all $i$,  
    $$\oH^{\CS}_{i}(P_1, (S_Z(M)_k  /S_Z(M)_{k+1}) \otimes \Omega_{n-2, m} \otimeshat (\chi \boxtimes \pi) \otimes\delta_{P_1}^{-1}) = 0.$$
\end{lemt}

\begin{proof}
    The proof is similar to that of Lemma \ref{Lem: GL closed grading vanish}.
    It suffices to prove the statement for
    $$
    \oH^{\CS}_{i}(L_1, \oH_{j}(\n_1, \BC) \otimes \lambda_{1,0} \otimes F_k \otimes \Omega_{n-2, m} \otimeshat (\chi \boxtimes \pi) \otimes\delta_{P_1}^{-1}).
    $$
    Let $\xi_j \boxtimes \beta_j$ be an irreducible $\GL(X_1) \times G(W_{n-2})$ constituent of $\oH_{j}(\n_1, \BC)$.
    We can further reduce to
    \begin{multline*}
    \oH^{\CS}_{i}(L_1, (\xi_j \boxtimes \beta_j) \otimes \lambda_{1,0} \otimes F_k \otimes \Omega_{n-2, m} \otimeshat (\chi \boxtimes \pi) \otimes\delta_{P_1}^{-1})
    \cong \\
    \bigoplus_{p+q=i} \oH^{\CS}_{p}(\GL(X_1), \xi'_j \otimes \chi \otimes F_k) \otimeshat 
    \oH^{\CS}_{q}(G(W_{n-2}), \beta'_j \otimes \Omega_{n-2, m} \otimeshat \pi).
    \end{multline*}
    where $\beta'_{j}$ (resp. $\xi'_{j}$) denotes $\beta_{j}$ (resp. $\xi_{j}$) twisted by a certain inessential character.
    Since $\xi'_{j} \otimes \chi$ is fixed, there must exist a constant $K_{i,j}$, such that for all $k > K_{i,j}$, 
    the $\BR^{\times}_{>0}$\mbox{-}component of $\GL(X_1)$ acts on $\xi'_{j} \otimes \chi  \otimes F_k$ via a nontrivial character.
    Thus, the lemma follows from a similar argument as in the proof of Lemma \ref{Lem: GL closed grading vanish}. 
\end{proof}

\begin{lemt}\label{Lem: Type I closed orbit}
    For all $i \geq 0$,
    $\oH^{\CS}_{i}(P_1, S_Z(M) \otimeshat \Omega_{n-2, m} \otimeshat (\chi \boxtimes \pi) \otimes \delta_{P_1}^{-1})$ is a Casselman--Wallach representation of $H_m$.
\end{lemt}

\begin{proof}
    It follows from Lemmas \ref{Lem: homological filtration} and \ref{Lem: Type I grading closed}.
\end{proof}

\begin{proof}[Proof of Proposition \ref{Prop: Type I principal induction}]
    Using Lemma \ref{Lem: Hausdorff LES}, the proposition follows from Lemmas \ref{Lem: Type I open orbit} and \ref{Lem: Type I closed orbit}.
\end{proof}

We are now ready to prove Theorem \ref{Thm: principal series CW} for the Type I case.

\begin{proof}[Proof of Theorem \ref{Thm: principal series CW}: Type I case]
Let $(W_n, V_m)$ be as before, and denote by $\t_{W}$ (resp. $\t_{V}$) the Witt tower containing $W_n$ (resp. $V_m$).
We prove that Theorem~\ref{Thm: principal series CW} holds for all dual pairs $(G(W), H(V))$ with $W \in \t_{W}$ and $V \in \t_{V}$. The proof proceeds by induction on the split rank $r_W$ of $W$. 

The base case is the compact dual pair, namely when $G(W)$ is compact. In this case, all higher-degree homology groups vanish. 
For the $\oH^{\CS}_0$ term, Hausdorffness is automatic, and the finite length property follows from Howe's original work \cite{Howe1989Remarks} and \cite[Section 3]{Howe1989TranscendingClassicalInvTheory}.

For the induction step, fix $W \in \t_{W}$ with $r_{W} \geq 1$. Assume that for all $V_m \in \t_{V}$, Theorem \ref{Thm: principal series CW} holds for dual pairs $(G(W'), H(V_m))$, where $W' \in \t_W$ has split rank $r_{W'} = r_{W} -1$. 
By Proposition~\ref{Prop: principal series implies all}, it follows that Theorem~\ref{ThmA} holds for all pairs $(G(W'), H(V_m))$, so the hypothesis of Proposition~\ref{Prop: Type II principal induction} is satisfied. 
Therefore, via induction by stages, the induction step follows from Proposition~\ref{Prop: Type II principal induction}. The theorem follows.
\end{proof}

\printbibliography

\end{document}